\pdfoutput=1
\documentclass[11pt]{article}
\usepackage{geometry}                
\usepackage[parfill]{parskip}    
 \usepackage{color} 
\usepackage{graphicx}
\usepackage{amssymb}
\usepackage{amsmath}
\usepackage{amsthm}
\usepackage{epstopdf}
\usepackage{youngtab}

\newcommand\blfootnote[1]{%
  \begingroup
  \renewcommand\thefootnote{}\footnote{#1}%
  \addtocounter{footnote}{-1}%
  \endgroup
}

\title{Apolarity for determinants and permanents of generic symmetric matrices}
\author{Sepideh Shafiei\\ \footnotesize{Department of Mathematics, Northeastern University, 
Boston, MA 02115, USA}}

\newtheorem{thm}{Theorem}[section]
\newtheorem{lem}[thm]{Lemma}
\newtheorem{cor}[thm]{Corollary}
\newtheorem{prop}[thm]{Proposition}
\newtheorem*{ack}{Acknowledgment}

\theoremstyle{definition}
\newtheorem{ex}[thm]{Example}

\theoremstyle{definition}
\newtheorem{defi}[thm]{Definition}

\theoremstyle{definition}
\newtheorem*{Notation}{Notation}

\theoremstyle{definition}
\newtheorem*{rmk}{Remark}

\newtheorem*{example}{Example}

\theoremstyle{definition}
\newtheorem{remark}[thm]{Remark}

\begin{document}
\maketitle
\blfootnote{sepideh.shafiee@gmail.com}

\begin{abstract} 
We show that the apolar ideal to the determinant of a generic symmetric matrix is generated in degree two, and the apolar ideal to the permanent of a generic symmetric matrix is generated in degrees two and three. In each case we specify the generators of the apolar ideal. As a consequence, using a result of K.~Ranestad and F. O. Schreyer, we give lower bounds to the cactus rank and rank of each of these polynomials. We compare these bounds with those obtained by J. Landsberg and Z. Teitler. We also determine for some cases the irreducible $S_n$ invariants involved in the apolar ideal to the permanent of a symmetric matrix. 
\end{abstract}

\section{Introduction}

This paper is a sequel and companion to \cite{Sh}. However, it can be read independently. In this paper we determine the annihilator ideals of the determinant and the permanent of an $n\times n$ generic symmetric matrix $X$. Here annihilator is meant in the sense of the apolar pairing, i.e. Macaulay's inverse system. In section~two, we review the doset basis of the space of the $t\times t$ minors of $X$. In section~three, we determine the generators of the apolar ideal to the determinant and permanent of $X$ (Theorems \ref{main-theorem-det-symmetric} and \ref{thm:perm-sym-main-theorem}). In section four, we apply our results  to find a lower bound for the scheme/cactus rank of the determinant and permanent of the generic symmetric matrix (Theorems \ref{thm:RS-rank-symm-det} and \ref{thm:RS-rank-symm-perm}). In section five, we give a representation-theoretical explanation of the degree three generators of the ideal $\mathrm{Ann}(\mathrm{perm}(X))$ in Lemma~\ref{lem:rep explanation 6}.                       
\vspace{0.2 in}

Let  $\sf k$ be an infinite field of characteristic zero or characteristic $p>2$, and let$X=(x_{ij})$ be an $n\times n$ symmetric matrix whose entries are the  $\frac{n(n+1)}{2}$ distinct variables $x_{ij}$, $1\leq i \leq j\leq n$. The determinant and the permanent of $X$ are polynomials of degree $n$. Let $R^s=\sf k$$[ x_{ij}]  (i\leq j),$ be the polynomial ring and $S^s=\sf k$$[y_{ij}]  (i\leq j),$ be the ring of differential operators associated to $R^s$, and let $R^s_k$ and $S^s_k$ denote the degree-$k$ homogeneous summands. $S^s$ acts on $R^s$ by differentiation, denoted by $S^s\circ R^s$.\par

\vspace{.2in}

\begin{defi}\label{divided-power-definition}(\cite{IK}, Appendix A) Let $\sf{k}$ be a field of arbitrary characteristic. Let $R=\sf {k}$$[z_1,\cdots,z_r]=\oplus_{j\ge 0}R_j$, Let $\mathcal D$ be the graded dual of $R$, i.e.

$$
\mathcal D=\oplus_{j\ge 0}Hom_k(R_j,k)=\oplus_{j\ge 0} \mathcal D_j.
$$

We consider the vector space $R_1$ with the basis $z_1,\cdots,z_r$ and the left action of $GL_r(k)$ on $R_1$ defined by $A z_i=\sum_{j=1}^{j=r}A_{ij}z_j$. Since $R=\oplus_{j\ge 0}Sym^{j}R_1$ this action extends to an action of $GL_r(k)$ on $R$. By duality this action determines a left action of $GL_r(k)$ on $\oplus_{j\ge 0} \mathcal D_j$. We denote by $z^U=z_1^{u_1}\cdots z_r^{u_r}, |U|=u_1+\cdots+u_r=j$ the standard monomial basis of $R_j$. 
We denote by
$$
Z^U=Z_1^{[u_1]}\cdots Z_r^{[u_r]}
$$

the basis of $\mathcal D_j$ dual to the basis $\{z^U:|U|=j\}$. We call these elements divided power monomials. We call the elements of $\mathcal D_j$ divided power forms, and the elements of $\mathcal D$ divided power polynomials. We extend the definition of $Z^{[U]}$ to multi-degrees $U=(u_1,\cdots,u_r)$ with negative components by letting $Z^{[U]}=0$ if $u_i<0$ for some $i$.

We define a ring structure on $\mathcal D$ by setting the product of two monomials to be 
$$
Z^{[U]}\cdot Z^{[V]}={{U+V}\choose{U}}Z^{[U+V]},
$$

where ${{U+V}\choose{U}}$ is a product of binomial coefficients. This is extended by linearity and gives $\mathcal D$ a structure of a $\sf{k}$-algebra.
\end{defi}
\vspace{.2in}

\begin{defi}\label{def:divided-power-IK-co}

(\cite{IK}, page 267) Let $\mathcal D$ be the divided power ring. Let $L=a_1x_1+\cdots+a_rx_r\in \mathcal D_1$. The divided power $L^{[j]}$ is defined as 
$$L^{[j]}=\sum_{j_1+\cdots +j_r=j} a_1^{j_1} \cdots a_r^{j_r} x_1^{[j_1]} \cdots x_r^{[j_r]}.$$

\end{defi}

Let $S^s=\sf k$$[y_{ij}]$ and $R^s=\sf k$$[x_{ij}]$. $S^s$ acts on $R^s$ by contraction as follows:

\begin{equation}\label{eq:contraction-co}
 ({y_{ij}})^k \circ_{co} (x_{uv})^\ell=\begin{cases}
  x_{uv}^{\ell-k} &  \text{if $(i,j)=(u,v)$ and $k\leq l$},\\
   0 &  \text{otherwise}.
\end{cases}
\end{equation}

This action extends multilinearly to the action of $S^s$ on $R^s$.

\par

\vspace{.2in}

\begin{remark}
For a field $\sf{k}$ with $\mathrm {char}\sf{k}=0$ or greater than the degree of the polynomial $F$, the contraction action on the divided power ring is an analogue of the partial differential operator action on the usual polynomial ring. If we consider the usual polynomial ring but find the apolar ideals using the contraction instead of differentiation, which is an unusual choice, the answer will be less regular compared to section three.  Taking the contraction yields information about writing $\det (X)$ as the sum of divided powers, not usual powers. Hence both the Hilbert function and generators of the apolar ideal are different and less regular for contraction versus differentiation.
\end{remark}

\vspace{.2in}
\begin{defi}\label{ann-length-degree-newdefi}

To each degree-$j$ homogeneous element, $F\in R^s_j$ we can associate the ideal $I=\mathrm {Ann} (F)$ in $S^s=\sf k$$[y_{ij}]$ consisting of polynomials $\Phi$ such that $\Phi\circ F=0$. We call $I=\mathrm {Ann} (F)$, the apolar ideal of $F$;  and the quotient algebra $S^s/\mathrm {Ann} (F)$ the apolar algebra of $F$. If $h\in S^s_k$ and $F \in R^s_n$, then we have $h\circ F\in R^s_{n-k}$.\par
Let $F \in R^s$, then $\mathrm {Ann} (F) \subset S^s$ and we have  
 
 \begin{equation*}
{(\mathrm {Ann}(F))}_{k}=\{ h \in S^s_k|h \circ F=0\}.
 \end{equation*}
 
We define the $\deg(\mathrm {Ann}(F))$ to be the length of $S^s/Ann(F)$. Let $V$ be a vector subspace of $R^s$, and define
 
 \begin{equation*}
 \mathrm {Ann}(V)=\{ h \in S|h \circ F=0 \text{ for all } F\in V\}.
 \end{equation*}
 
 \end{defi}
 
 \begin{remark}

Let $\sf k$ be a field of characteristic $0$ or $p>i$. Let $\phi: (S^s_i,R^s_i)\rightarrow \sf k$ be the pairing $\phi(g,f)=g\circ f$,  and $V$ be a vector subspace of $R^s_k$. Let $V^\perp$ denote $\mathrm{Ann}(V)\cap S^s_k$. We have
 \begin{equation*}
 \dim_{\sf k} (V^\perp)=\dim_{\sf k} S^s_k-\dim_{\sf k} V.
 \end{equation*}
 
 \end{remark}

Let $F$ be a form of degree $j$ in $R^s$. We denote by  $\langle F \rangle _{j-k}$ the vector space $S^s_k \circ F \subset R^s_{j-k}$. 
\vspace{0.2in}

\begin{remark}(see \cite{IK}, Lemma 2.15)\label{remark:introIK} 
Let $F \in R^s$ and $\deg F=j$ and $k \leq j$. Then we have

 \begin{equation}\label{eq:introIK}
(\mathrm{Ann}(F))_k=\{h \in S^s_k| h \circ (S^s_{j-k}\circ F)=0\}=(\mathrm {Ann} (S^s_{j-k}\circ F))_k.
 \end{equation}

\end{remark}

\vspace{.2in}

We define the homomorphism $\xi:R^s\rightarrow S^s$ by setting $\xi(x_{ij})=y_{ij}$; for a monomial $v\in R^s$ we denote by  $\hat{v}=\xi(v)$ the corresponding monomial of $S^s$, and $Y=\xi(X)$.
\vspace{.2in}
\begin{remark} \label{remark:main-ann} (see \cite{Sh}, Remark 2.8)
Let $f= \sum_{i=1}^{i=k} \alpha_i v_i  \in R^s_n$, with $\alpha_{i}\in \sf k$ and with $v_i$'s linearly independent monomials. Then we have
\begin{equation*}
\mathrm {Ann}(f) \cap S^s_n=\langle \{\alpha_j \hat{v_1}-\alpha_1\hat{v_j}\mid 2\leq j \leq n\},\langle v_1,\ldots,v_k \rangle ^\perp \rangle,
\end{equation*}
where $\langle v_1,\ldots ,v_k \rangle ^\perp=\mathrm {Ann}(\langle v_1,\ldots ,v_k \rangle)\cap S^s_n$.
\end{remark}

Denote by $\mathfrak A_X=S^s/(\mathrm {Ann} (\det (X))$ the \emph{apolar algebra} of the determinant of the matrix $X$. Recall that the Hilbert function of $\mathfrak A_X$ is defined by $H(\mathfrak A_X)_i=\dim_{\sf{k}} (\mathfrak A_X)_{i}$ for all $i=0,1,\ldots \,.$ 
\vspace{0.2in}


\subsection{Summary of main results}

\begin{itemize} 

\item We specify the Hilbert sequence corresponding to the apolar algebras in the usual contraction-divided powers pairing, or, equivalently, in the differentiation-usual powers pairing of the following homogeneous polynomials

\begin{itemize}
\item Determinant of a generic symmetric matrix (Table 1). This uses the Conca Theorem \ref{thm:doset-Conca}.
\item Permanent of the generic symmetric matrix (Table 2).
\end{itemize}

\item We specify the generators of the apolar ideal in each of the following cases:
\begin{itemize} 
\item Determinant of a generic symmetric $n\times n$ matrix $X$. This ideal is generated by certain $2\times 2$ permanents of $Y=\xi (X)$, certain degree two trinomials that are Hafnians of $4\times 4$ symmetric submatrices of $Y$ and some monomials (Proposition \ref{prop:gen-det-symm}). In particular, this ideal is generated in degree two (Theorem \ref{main-theorem-det-symmetric}).
\item Permanent of a generic symmetric $n\times n$ matrix $X$. This ideal is generated by certain $2\times 2$ minors of $Y=\xi (X)$, certain degree three polynomials corresponding to $6\times 6$ Hafnians of $Y$ and some degree two monomials (Proposition \ref{prop:gens-deg2-perm-symm}, Lemma \ref{prop:gens-deg3-perm-symm}). In particular, this ideal is generated in degrees two and three (Theorem \ref{thm:perm-sym-main-theorem}).
\end{itemize}

\item  In each of the above cases the proof has several main steps:
\begin{itemize}   
\item Let $I$ be the the apolar ideal. We identify the dual module to $S/I$, so we determine $S_i\circ F$, where $F$ is the invariant.  
\item  For the determinant we determine $I_2$, where $I$ is the apolar ideal; and for the permanent we determine $I_2$ and $I_3$, and let $I^+=(I_2,I_3)$. 
\item In the case of the determinant we show that  $(I_2)_k$ is the full perpendicular space in $S_k$ to $S_{n-k}\circ \det (X)$. And in the case of permanent we show that $(I^+)_k$  is the full perpendicular space in $S_k$ to $S_{n-k}\circ \mathrm{perm}(X)$.
 \end{itemize} 
   
Of these steps, the last is the hardest and we use a triangularity method. For the determinant of a symmetric matrix we show that the acceptable monomials which are not the leading term of a Conca doset minor are the initial monomial of the generators of the ideal $(I_2)$ in the reverse lexicographic order (Proposition 3.8), and for the permanent of symmetric matrix we use a similar triangularity method (Proposition 3.19).

 \item We apply these results to give a lower bound for:
\begin{itemize}
\item Cactus rank of the determinant of a generic symmetric $n\times n$ matrix (Theorem 4.5).
\item Rank of the determinant of the generic symmetric $n\times n$ matrix (Proposition 4.8).
\item Cactus rank of the permanent of a generic symmetric $n\times n$ matrix (Theorem 4.6).
\end{itemize}

\item We give a Gr\"{o}bner basis for the apolar ideal of the determinant of a generic symmetric matrix (Theorem 3.12).

\item We give a representation-theoretical explanation of the degree three generators of the ideal $\mathrm{Ann}(\mathrm{perm}(X))$ in Lemma \ref{lem:rep explanation 6}.

\end{itemize}
\vspace{.2in}
\begin{ack} I am deeply grateful to my advisor Prof. A. Iarrobino whose help, stimulating ideas and encouragement helped me in working on this problem and writing this paper. I am very thankful to Prof. Z. Teitler for suggesting this problem and his helpful comments and also Prof. A. Conca and Prof. L. Smith for their valuable comments and suggestions. I also gratefully acknowledge support in summer 2012 from the Ling-Ma fellowship of the Department of Mathematics at Northeastern University.  \par\medskip
\end{ack}

\section{Doset basis for the space of $k\times k$ minors}
\label{section: Dimension}

We recall the definition of doset minors and the Gr\"{o}bner basis for the determinantal ideal of a generic symmetric matrix.

\vspace{.2in}
\begin{defi} \label{def:doset-Conca}(see \cite{CON}) Let $H$ be the set of all subsequences $(a_1,\ldots,a_t)$ of  $(1,\ldots,n)$, i.e., $a_1<a_2<\dots<a_t$. Let $a,b\in H$. We define on $H$ the partial order

\begin{equation*}
a=(a_1,\ldots,a_t) \leq b=(b_1,\ldots,b_r) \iff r\leq t, \text{ and } a_i\leq b_i \text{ for } i=1,\ldots,r.
\end{equation*}

We denote by $[a_1,\ldots,a_t| b_1,\ldots,b_t]$ the minor $\det(X_{a_ib_j}), 1\leq i,j, \leq t$ of $X$. Since $X$ is symmetric it is clear that $[a,b]=[b,a]$. A minor $[a_1,\ldots,a_t|b_1,\ldots,b_t]$ of $X$ with $a\le b$ in $H$ is called a doset minor.

\end{defi}
\vspace{.2in}

Let $\tau$  be a diagonal term order on $R^s=\sf k$$[x_{ij}]$: that is the initial term of every doset minor $[a_1,\ldots,a_s|b_1,\ldots,b_s]$ is $\prod_{i=1}^s x_{a_ib_i}$. For instance, we can consider the lexicographic order induced by the variable order
\begin{equation*}
x_{11}\geq x_{12} \geq \ldots \geq x_{1n}\geq x_{22} \geq \ldots \geq x_{2n} \geq \ldots \geq x_{n-1 n}\geq x_{nn}.
\end{equation*}

\begin{thm}\label{thm:doset-Conca}\textbf{(Conca)}[\cite{CON}, Theorem 2.9]
Let $M_t(X)$ be the ideal generated by the $t$-minors of $X$. The set of the doset t-minors is a Gr\"{o}bner basis for $M_t(X)$ with respect to $\tau$.
\end{thm}
\vspace{.2in}
\begin{defi}\label{def:young-tablaux}
A Young tableau of shape $(r_1,\ldots,r_u)$ is an array of positive integers $A=(a_{ij})$, with $1\leq i\leq u$, $1\leq j \leq r_i$, and $r_1 \geq \ldots \geq r_u$. Such a tableau is said to be semi-standard if  the numbers in each row is strictly increasing from left to right, and the numbers in each column are in non-decreasing order from top to bottom (that is $a_{i,j}< a_{i,j+1}$ for all $i=1,\ldots,u, j=1,\ldots,r_i-1$ and $a_{i,j} \leq a_{i+1,j}$ for all $i=1,\ldots,u-1, j=1,\ldots,r_{i+1})$.
\end{defi}
\vspace{.2in}
\begin{example} An example of a semi-standard Young tableau of shape $(4,3,2,2,1)$ filled with the numbers $\{1,2,3,4\}$ is 
\vspace{0.2in}

$$\young(1234,134,23,24,2)$$
\end{example}

\vspace{0.2in}
\begin{defi} \label{def:Lattice-paths}
A path composed of horizontal and vertical line segments in the $x$-$y$ plane from $(0,0)$ to $(n,n)$ with steps $(0,1)$ and $(1,0)$ is called a \emph{lattice path} of order $n$. A lattice path that never rises above the line $y=x$, is called a \emph{Dyck path} of order $n$. The \emph{corners} of the Dyck path are the points on the path where the direction of the path changes from horizontal to vertical or vice versa.
\end{defi}
\vspace{.2in}

The total number of Dyck paths of order $n$ is given by the Catalan number (\cite{ST}, Volume 2, page 221, Exercise 6.16 h)
\begin{equation*}
c_n=\frac{1}{n+1} {2n\choose n}.
\end{equation*}

\vspace{.2in}
\begin{lem}\label{lem:Narayana-symm-minors}
The dimension of the vector space spanned by the $t\times t$ minors of an $n\times n$ generic symmetric matrix is equal to the number of doset t-minors of the $n\times n$ symmetric matrix. This is equal to the number of semi-standard fillings of a Young tableau of shape $(t,t)$ with the numbers $\{1,...,n\}$, which is equal to the Narayana number $$N(n+1,t+1)={n+1\choose t+1}{n+1 \choose t}/(n+1).$$
\end{lem}

\begin{proof} The first statement is true by Conca's Theorem \ref{thm:doset-Conca}. To show the second statement one can view the count of the Conca doset $t$-minors as giving the coordinates in the $x$-$y$ plane of $t$ points. Then the count is of segmented paths with $t$ interior corners lying on or below the diagonal,   beginning at $(0,0)$ and ending at $(n+1,n+1)$. That is the number of Dyck $(n+1)$-paths with exactly $t$ vertices, which is given by the Narayana numbers as ${n+1\choose t+1}{n+1 \choose t}/(n+1)$. See \cite{ST}, Volume 2, page 237, Exercise 6.36 a, with a change of $+1$ to $(1,0)$ and $-1$ to $(0,1)$. 

\end{proof}

 \vspace{.2in}
\begin{lem}\label{lem:Scircledet-symm} Let $M_{n-k}(X)$ be the vector space spanned by the $(n-k)$-minors of the generic symmetric matrix $X$. Then
 \begin{equation}\label{eq::Scircledet-symm}
S^s_{k}\circ(\det(X))=M_{n-k}(X) \subset R^s.
 \end{equation}
 \end{lem}
 \begin{proof}
To show the inclusion
  \begin{equation*}
S^s_{k}\circ(\det(X))\subset M_{n-k}(X) \subset R^s,
 \end{equation*}
 
 we use induction on $k$. For $k=1$, the above inclusion is easy to see. Now assume that the above inclusion holds for $k-1$, i.e $S^s_{k-1}\circ \det(X) \subset M_{n-(k-1)}(X)$, and we want to show that it is true for $k$. We have $$S^s_k\circ \det(X)=S^s_1S^s_{k-1}\circ \det(X)\subset S^s_1\circ M_{n-k+1}(X)\subset M_{n-k}(X).$$
 
 Now we show  the opposite inclusion
 \begin{equation*}
S^s_{k}\circ(\det(X))\supset M_{n-k}(X) \subset R^s,
 \end{equation*}

Let $M_{\widehat I,\widehat J}(X), I=\{i_1,\ldots ,i_k\}, J=\{j_1,\ldots j_k\}, 1\le i_1< i_2 < \cdots < i_k\le n, 1\le j_1< j_2 < \cdots < j_k\le n$ be the $(n-k)\times (n-k)$ minor of $X$ one obtains by deleting the $I$ rows and $J$ columns of $X$. Let $\Delta=\Delta_{(I,J)}=I\cap J.$ Let $M_{(I,J)-\Delta}(Y)$ be the sub matrix of Y with the rows $I-\Delta$, and the columns $J-\Delta$. We claim

$$M_{\widehat{I},\widehat{J}}(X)=\pm c \bigg(\big( \prod_{i\in\Delta}y_{ii} \big) \det(M_{(I,J)-\Delta}(Y))\bigg)  \circ \det (X)$$ where $c\not=0 \in \sf{k}$. 

To prove this claim we use induction on $k=|I|=|J|$, the cardinality of the sets $I$ and $J$.  First we show the claim is true for $k=1$. Let $I=\{i_1\}$ and $J=\{j_1\}$. We have two cases

I. $i_1=j_1$ so $y_{i_1j_1}$ is a diagonal element and we have

\begin{equation*}
M_{\widehat{I},\widehat{J}}(X)=y_{i_1j_1}\circ (\det(X)).
\end{equation*}

II. $i_1\neq j_1$ so we have

\begin{equation*}
y_{i_1j_1}\circ (\det(X))=2M_{\widehat{I},\widehat{J}}(X).
\end{equation*}

So for $k=1$ the claim holds. Next assume that the claim holds for every $I$ and $J$ with $|I|=|J|=k-1$ and we want to show that the claim is also true for $I$ and $J$ with $|I|=|J|=k$. Let $I=\{i_1,\ldots,i_k\}$ and $J=\{j_1,\ldots,j_k\}$. Let $I'=I-\{i_1\}$ and $J'=J-\{j_1\}$. We have $|I'|=|J'|=k-1$ so by the induction assumption we have 

\begin{equation*}
M_{\widehat{I'},\widehat{J'}}(X)=\pm c\bigg( \big(\prod_{i\in\Delta_{(I',J')}}y_{ii} \big) \det(M_{(I',J')-\Delta}(Y))\bigg)  \circ \det (X)
\end{equation*}

By writing the block expansion of the determinant using row $i_1$ or column $j_1$ for $M_{\widehat{I},\widehat{J}}$, we get 

$$M_{\widehat{I},\widehat{J}}(X)=\pm c \bigg(\big( \prod_{i\in\Delta_{(I,J)}}y_{ii}\big) \det(M_{(I,J)-\Delta}(Y))\bigg)  \circ \det (X),$$ where $c\not=0 \in \sf{k}$. Hence $M_{\widehat{I},\widehat{J}}(X)\in S^s_{n-k} \circ (\det(X))$.

 \end{proof}

\begin{cor} \label{cor:degree-catalan}Let $X$ be a generic symmetric $n\times n$ matrix.  Then $$H_t(S^s/\mathrm {Ann}(\det(X))=N(n+1,t+1),$$ and $\deg(\mathrm {Ann}(\det(X))$ (Definition \ref{ann-length-degree-newdefi})  is  the Catalan number $C_{n+1}=\frac{1}{n+2}{2n+2 \choose n+1}$.
\end{cor}

\begin{proof}
Let $M_t$ be the space of $t \times t$ minors of a symmetric $n\times n$ matrix. Note that $N(n+1,n+1)=1$, and $H(S^s/\mathrm {Ann}(\det(X)))_0 = 1$, so by Lemma \ref{lem:Scircledet-symm} we have
\begin{equation*}
\deg(\mathrm {Ann}(\det(X))=\sum_{t=0}^{t=n} N(n+1,t+1)=C_{n+1}.
\end{equation*}
Thus the $\deg(\mathrm {Ann}(\det(X))$ will be the total number of Dyck paths, below or meeting the diagonal through the $(n+1)\times (n+1)$ grid, which is given by the Catalan number $C_{n+1}=\frac{1}{n+2}{2n+2 \choose n+1}$. (See \cite{ST}, Volume 2, page 237, Exercise 6.36 a)

\end{proof}

\begin{table}[h]
\begin{center}
\caption{\small{The Hilbert sequence of the Artin algebra $S^s/I$, where $I$ is the annihilator of the determinant of the generic symmetric matrix}}\label{table:determinant-symmetric-hilbert}

\begin{tabular}{l*{9}{c}r}

\hline
degree & 0 & 1 & 2 & 3 & 4 & 5 & 6 & 7 & 8  \\
\hline
n=2 & 1 & 3 & 1  \\
\hline
n=3           & 1 & 6 & 6 &  1 \\
\hline
n=4 & 1 & 10 & 20 & 10 & 1 \\
\hline
n=5 & 1 & 15 & 50 & 50 & 15 & 1 \\
\hline
n=6 & 1 & 21 & 105 & 175 & 105 & 21 & 1\\
\hline
n=7 & 1 & 28 & 196 & 490 & 490 & 196 & 28 & 1\\
\hline
n=8 & 1 & 36 & 336 & 1176 & 1764 & 1176 & 336 & 36 & 1\\

\end{tabular}
\end{center}
\end{table}


\section{Generators of the apolar ideal } 
\label{Section:symmetric-Generators of the apolar ideal } 

In section \ref{subsection:Apolar ideal of the determinant} we determine the generators of the apolar ideal of the determinant of the $n\times n$ generic symmetric matrix. In section \ref{subsection:Apolar ideal of the permanent-symm} we determine the generators of the apolar ideal of the permanent of the $n\times n$ generic symmetric matrix.

\vspace{0.2in}

\begin{Notation}(\cite{IKO})
Let $F_{2m}\subset \mathfrak{S}_{2m}$ be the set of all permutations $\sigma$ satisfying the following conditions:

(1) $\sigma(1)<\sigma(3)<\cdots<\sigma(2m-1)$

(2) $\sigma(2i-1)<\sigma(2i)$ for all $1\leq i \leq m$

We denote by $Hf(X)$ the hafnian of a generic symmetric $2n\times 2n$ matrix $X$, which is defined by

\begin{equation}\label{eq:hafnian-definition}
Hf(X)=\sum_{\sigma\in F_{2n}} x_{\sigma(1)\sigma(2)}x_{\sigma(3)\sigma(4)}\cdots x_{\sigma(2n-1)\sigma(2n)}
\end{equation}
\end{Notation}


\subsection{Apolar ideal of the determinant}
\label{subsection:Apolar ideal of the determinant}

In this subsection we determine the apolar ideal of the determinant of the $n\times n$ generic symmetric matrix, and we will show that it is generated by its degree two elements. We first determine the ideal in degree two (Proposition \ref{prop:gen-det-symm}). We then show that the generators in degree two generate the annihilator ideal in degree $n$ (Proposition \ref{prop:main-det-symm-n}). Then we show that the elements of degree two generate the ideal in each degree $k$ for $2\leq k\leq n$. A key step is to use triangularity to show that these degree two generators, generate all of the apolar ideal (Lemma \ref{lem:det-symm-unacceptable-traiangularity} and Proposition \ref{prop:triangularity-det}).  This leads to our main result (Theorem \ref{main-theorem-det-symmetric}).
\vspace{0.2in}
\begin{Notation}
Let $X$ be the generic symmetric $n\times n$ matrix. Recall that $Y=\xi(X)$ is the generic $n\times n$ symmetric matrix in the variables of $S^s$. The \emph{unacceptable monomials} of degree $k$ in $S^s_k$ are monomials which do not divide any term of the determinant of $Y$. We denote the set of degree $k$ unacceptable monomials by $U_k$. A monomial that divides some term of the determinant is called an \emph{acceptable} monomial.
\end{Notation}
\vspace{.2in}

\begin{prop}\label{prop:gen-det-symm}
For the $n\times n$ generic symmetric matrix $X=(x_{ij})$, and $Y=\xi(X)$, $\mathrm{Ann}(\det(X))\subset S^s$, includes the following degree 2 polynomials:

(a) The unacceptable monomials of the form $y_{ii}y_{ij}$ for all $1\leq i,j \leq n$ ($i=j$ is allowed). The number of these monomials is $n^2$.

(b) All the diagonal $2\times 2$ binomials of the form $y_{ij}^2+2y_{ii}y_{jj}$ ($i\neq j$). The number of these binomials is $n\choose 2$.

(c) All the $2\times 2$ permanents with one diagonal element, i.e. $y_{ji}y_{il}+y_{jl}y_{ii}$ ($i$, $j$, and $l$ are pairwise distinct). The number of these binomials is $n\cdot{{n-1}\choose 2}$.

(d) The hafnians of all symmetrically chosen $4\times 4$ submatrices of $Y=\xi(X)$. The number of these trinomials is $n\choose 4$.

\end{prop}

\begin{proof}
We have $\det(X)=\sum _ {\sigma \in S_n} Sgn(\sigma) \Pi x_{i,\sigma(i)}$. First we show that monomials of type (a) are in $\mathrm {Ann}(\det(X))$. By symmetry we have

$$y_{ii}y_{ij}\circ \det(X)=0 \text{ (where } j\geq i),$$
$$ y_{ii}y_{ji}\circ\det(X)=0 \text{ (where } j\leq i).$$

Next we want to show that binomials of type (b) are in $\mathrm {Ann}(\det(X))$.

Let $P=2y_{ii}y_{jj}+y_{ij}^2$.There are $n!$ terms in the expansion of the determinant. If a  term doesn't contain the monomial $x_{ii}x_{jj}$ or the monomial $x_{ij}^2$ then the result of the action of $P$ on it will be zero. 
Let $\sigma_{1}$ be a permutation having $x_{ii}$ and $ x_{jj}$ respectively in its $i$-th and $j$-th place. Corresponding to $\sigma_{1}$ we also have a permutation $\sigma_{2}=\tau \sigma_{1}$, 
 where $\tau=(i,j)$ is a transposition and $sgn (\sigma_{2})=sgn (\tau\sigma_{1})=-sgn( \sigma_{1})$. Thus, corresponding to each positive term in the determinant which contains the monomial $x_{ii} x_{jj}$ or the monomial $x_{ij}^2$ we have the same term with the negative sign, thus the resulting action of the binomial $P$ on $\det (X)$ is zero.
 
To show that the binomials of type (c) are in the annihilator ideal we can use the same proof as we used for the binomials of type (b).

Next we want to show that any $4\times 4$ hafnian of $Y$ annihilates the determinant of an $n\times n$ symmetric matrix $X$. This is easy to check for $n=4$.
So let $n\ge 4$. Let $W$ be a $4\times 4$ submatrix of $Y$, involving the rows and the columns $i_1,i_2,i_3$ and $i_4$.

\begin{center}  $ W= \left(%
\begin{array}{cccc}
 
  y_{i_1i_1} & y_{i_1i_2} & y_{i_1i_3} & y_{i_1i_4}\\
  y_{i_2i_1} & y_{i_2i_2} & y_{i_2i_3} & y_{i_2i_4}\\
  y_{i_3i_1} & y_{i_3i_2} & y_{i_3i_3} & y_{i_3i_4}\\
  y_{i_4i_1} & y_{i_4i_2} & y_{i_4i_3} & y_{i_4i_4}\\
  \end{array}%
\right)$

\end{center}

By Equation \ref{eq:hafnian-definition} the hafnian of $W$ is 

$$H= \mathrm{Hf}(W)=y_{i_1i_2}y_{i_3i_4}+ y_{i_1i_3}y_{i_2i_4}+ y_{i_1i_4}y_{i_2i_3}.$$

If a term in the determinant does not contain the monomials $x_{i_1i_2}x_{i_3i_4}$ or  $x_{i_1i_3}x_{i_2i_4}$ or  $x_{i_1i_4}x_{i_2i_3}$, then $H$ annihilates it. If a term in the determinant contains one of the monomials  $x_{i_1i_2}x_{i_3i_4}$ or  $x_{i_1i_3}x_{i_2i_4}$ or  $x_{i_1i_4}x_{i_2i_3}$, then since these monomials do not appear in any other $4\times 4$ sub matrix of $X$, we can use the block expansion (cofactor expansion) of the determinant with the rows $i_1\dots i_4$,   and the proof is complete.

\end{proof}
\vspace{.2in}

We denote by $Q$ be set of the degree two elements of type (a), (b), (c) and (d) in Proposition~\ref{prop:gen-det-symm}, and by $\langle Q\rangle $ the vector subspace of $S^s$ spanned by $Q$. We denote by $Q_a$, $Q_b$,  $Q_c$ and  $Q_d$ the set of elements in $(a)$, $(b)$, $(c)$ and $(d)$ respectively.
\vspace{.2in}
\begin{lem}\label{lem:symm-det-dimV}
The set $Q$ is linearly independent and we have,
  \begin{equation*}
 \dim \langle Q\rangle=n^2+{n\choose 2}+n\cdot{{n-1}\choose 2}+{n\choose 4}.
 \end{equation*}
 
 \end{lem}
 
 \begin{proof}
Each of the four subsets is linearly independent from the span of the other three, since they involve different monomials. So it suffices to show that each subset is linearly independent. The subset $Q_a$ is linearly independent since the monomials in $Q_a$ form a Gr\"{o}bner basis for the ideal they generate. The subsets $Q_b$ and $Q_c$ are linearly independent since by choosing two elements of the matrix, where at least one element is diagonal, we have a unique $2\times 2$ minor. The subset $Q_d$ is linearly independent since the monomials that appear in a hafnian of a $4\times 4$ symmetric submatrix of $Y$, do not appear in the hafnian of any other $4\times 4$ symmetric submatrix of $Y$. Hence the set $Q$ is linearly independent and the dimension of the vector space $\langle Q\rangle $ is $n^2+{n\choose 2}+n\cdot{{n-1}\choose 2}+{n\choose 4}$.

 \end{proof}

\vspace{.2in}

\begin{lem} \label{lem:det-symm-V-degree2}For the generic symmetric $n\times n$ matrix $X$, we have
\begin{equation*}
\langle Q\rangle=\mathrm{Ann}(M_2(X))\cap S^s_2=(\mathrm{Ann}(\det X))_2
\end{equation*}

\end{lem}

\begin{proof}
By the Lemma \ref{lem:Scircledet-symm} we have 

\begin{equation*}
\mathrm{Ann}(S^s_{n-2}\circ(\det(X)))= \mathrm{Ann}(M_2(X)).
\end{equation*}

By the Proposition \ref{prop:gen-det-symm} we have

\begin{equation*}
(\mathrm{Ann}(\det(X)))_2\supset \langle Q\rangle
\end{equation*}

By the Remark \ref{remark:introIK} we have

\begin{equation*}
(\mathrm{Ann}(\det(X)))_2=(\mathrm{Ann}(S^s_{n-2}\circ(\det(X))))_2=\mathrm{Ann}(M_2(X))\cap S^s_2.
\end{equation*}

Hence we have

\begin{equation*}
\langle Q\rangle \subset \mathrm{Ann}(M_2(X)).
\end{equation*}

On the other hand, using Lemma \ref{lem:symm-det-dimV} we have 

\begin{equation*}
\dim\langle Q\rangle=n^2+{n\choose 2}+n\cdot{{n-1}\choose 2}+{n\choose 4}={{\frac{n^2+n}{2}}+1\choose 2}-{\frac{{n+1\choose 2}{n+1\choose 3}}{n+1}}=\dim S^s_2-\dim M_2(X).
\end{equation*}

So we have
\begin{equation*}
\langle Q\rangle=\mathrm{Ann}(M_2)\cap S^s_2.
\end{equation*}

\end{proof}
\vspace{.2in}

\begin{lem}\label{lem:easydirectionk-det-dymm} Let $(Q)$ be the ideal generated by the vector space $\langle Q\rangle $. For $2\leq k\leq n$ we have
\begin{equation}
(Q)_k \subset \mathrm {Ann}(M_k(X))  \cap S^s_k.
\end{equation}

\end{lem}

\begin{proof} By Remark \ref{remark:introIK} and Lemma \ref{lem:Scircledet-symm} we have

$$Q \circ \det (X)=0  \Longleftrightarrow Q \circ (S^s_{n-2}\circ \det(X))=0 \Longleftrightarrow  Q \circ M_2(X)=0.$$

By Remark  \ref{remark:introIK} and Lemma \ref{lem:det-symm-V-degree2} we have 
$$\mathrm {Ann}(\det(X))) \cap S^s_2= Q.$$
Hence by Lemma \ref{lem:Scircledet-symm} we have

$$S^s_{k-2} (Q)\circ (S^s_{n-k} \circ \det (X))=S^s_{k-2}(Q) \circ M_k(X)=(Q)_k \circ M_k(X)=0.$$

\end{proof}

\vspace{.2in}
\begin{prop} \label{prop:main-det-symm-n}For $n\geq 2$ we have
\begin{equation}\label{eq:main-det-symm-n}
(Q)_n=\mathrm{Ann}(\det(X))\cap S^s_n.
\end{equation}

\end{prop}

\begin{proof}

One inclusion is given by Lemma \ref{lem:easydirectionk-det-dymm}.  To show that the other inclusion holds we use induction on $n$. For $n=2,3$ the equality is easy to see. We next show the equality (\ref{eq:main-det-symm-n}) for $n=4$. Here

\begin{center}  $X=(x_{ij})= \left(%
\begin{array}{cccc}
 
  a & b & c & d\\
  b & e & f & g\\
  c & f & h & i\\
  d& g & i & j\\
  \end{array}%
\right)$,

\end{center}

\begin{center}  $Y=(y_{ij})= \left(%
\begin{array}{cccc}
 
  A & B & C & D\\
  B & E & F & G\\
  C & F & H & I\\
  D & G & I & J\\
  \end{array}%
\right)$.

\end{center}

We have:

$$\det(X)=d^2f^2-2cdfg+c^2g^2-d^2eh+2bdgh-ag^2h+2cdei-2bdfi-$$$$2bcgi+2afgi+b^2i^2-aei^2-c^2ej+2bcfj-af^2j-b^2hj+aehj \in R^s_4.$$

If we denote the determinant in the divided power ring by $\det(X)_{Div}$ we have:

$$\det(X)_{Div}=4d^2f^2-2cdfg+4c^2g^2-2d^2eh+2bdgh-2ag^2h+2cdei-2bdfi-$$$$2bcgi+2afgi+4b^2i^2-2aei^2-2c^2ej+2bcfj-2af^2j-2b^2hj+aehj\in \mathcal D_4.$$

We use the divided powers and the contraction in the following proof. Using the Remark~\ref{remark:main-ann}, we let $\psi$ be a binomial in $\mathrm{Ann}(\det(X))\cap S^s_4$. 

\begin{equation*}
\psi=\alpha_{\sigma}(-1)^{\mathrm{sgn}{(\eta)}}y_{1\eta(1)}y_{2\eta(2)}y_{3\eta(3)}y_{4\eta(4)}-\alpha_{\eta}(-1)^{\mathrm{sgn}(\sigma)}y_{1\sigma(1)}y_{2\sigma(2)}y_{3\sigma(3)}y_{4\sigma(4)},
\end{equation*}

where $\sigma\neq \eta$ are two permutations of the set $\{1,2,3,4\}$, $\alpha_{\eta}$ is the coefficient of the monomial $y_{1\eta(1)}y_{2\eta(2)}y_{3\eta(3)}y_{4\eta(4)}$ in $\det(X)_{Div}$ and $\alpha_{\sigma}$ is the coefficient of  $y_{1\sigma(1)}y_{2\sigma(2)}y_{3\sigma(3)}y_{4\sigma(4)}$ in $\det(X)_{Div}$. The terms  $y_{1\eta(1)}y_{2\eta(2)}y_{3\eta(3)}y_{4\eta(4)}$ and $y_{1\sigma(1)}y_{2\sigma(2)}y_{3\sigma(3)}y_{4\sigma(4)}$ cannot have 3 common factors, since if they have 3 variables in common the fourth variable is forced and it contradicts our assumption $\sigma\neq \eta$ . Without loss of generality we can assume $\eta=id$. We have three different possibilities. 

(i) $y_{11}y_{22}y_{33}y_{44}$ and $y_{1\sigma(1)}y_{2\sigma(2)}y_{3\sigma(3)}y_{4\sigma(4)}$ have two common factors. Without loss of generality we can assume that $\sigma(1)=1$ and $\sigma(2)=2$. So we have  

\begin{equation*}
\psi=2AEHJ-(-1)^{sgn(\sigma)} AEy_{3\sigma(3)}y_{4\sigma(4)}=2AEHJ+AEI^2=AE(2HJ+I^2)\in (Q)_4
\end{equation*}

(ii) $y_{11}y_{22}y_{33}y_{44}$ and $y_{1\sigma(1)}y_{2\sigma(2)}y_{3\sigma(3)}y_{4\sigma(4)}$ have only one common factor. Without loss of generality we can assume that $\sigma(1)=1$, Since the only term in the determinant which has $a$ and does not have $e,h$ and $j$ is $2afgi$ we have

$$\psi=2AEHJ-AFGI=2AEHJ-AFGI+AF^2J-AF^2J=$$$$AJ(2EH+F^2)-AF(GI+FJ)\in (Q)_4,$$

\vspace{0.2in}

since we know that $2EH+F^2\in Q$ and $GI+FJ\in Q$.

(iii) $y_{11}y_{22}y_{33}y_{44}$ and $y_{1\sigma(1)}y_{2\sigma(2)}y_{3\sigma(3)}y_{4\sigma(4)}$ do not have  any common factor. We add and subtract a term which has a common factor with $y_{11}y_{22}y_{33}y_{44}$ and a common factor with $y_{1\sigma(1)}y_{2\sigma(2)}y_{3\sigma(3)}y_{4\sigma(4)}$. The reason that such a term exists in the determinant is that if we choose two elements, $\alpha$ and $\beta$ not in the same row or column, it is easy to see that we always have a term in the determinant containing $\alpha\beta$. On the other hand if we choose one variable from  $y_{11}y_{22}y_{33}y_{44}$, say $y_{11}$, there is always one variable in $y_{1\sigma(1)}y_{2\sigma(2)}y_{3\sigma(3)}y_{4\sigma(4)}$ which is not in the first row or column, since we only have three elements other than $y_{11}$ in the first row and column. So we can always choose a term in the determinant with at least one common factor with $y_{11}y_{22}y_{33}y_{44}$ and at least one common factor with $y_{1\sigma(1)}y_{2\sigma(2)}y_{3\sigma(3)}y_{4\sigma(4)}$. Then using the cases (i) or (ii) we have
\vspace{0.2in}

$$\psi=AEHJ-AFGI=2AEHJ-AFGI+AF^2J-AF^2J=$$$$AJ(2EH+F^2)-AF(GI+FJ)\in (Q)_4.$$

This completes the proof of the Equation \ref{eq:main-det-symm-n} for $n=4$. 

\vspace{0.2in}

Now let $n\ge 5$. By the induction assumption Equation \ref{eq:main-det-symm-n} holds for all integers $2\leq k\leq n-1$. Again we use the Remark \ref{remark:main-ann}. Let $\beta=\beta_1+\beta_2\in \mathrm{Ann}(\det(X))\cap S^s_4$. If the two terms $\beta_1$ and $\beta_2$ have a common factor $l$, i.e. $\beta_1=la_1$ and $\beta_2=la_2$, then $\beta=l(a_1+a_2)$ where $a_1$ and $a_2$ are of degree at most $n-1$. By  the induction assumption the proposition holds for the binomial $a_1+a_2$, i.e. $a_1+a_2 \in (Q)_{n-1}$  hence we have 
\begin{equation*}
\beta=l(a_1+a_2)\in l(Q)_{n-1}\subset (Q)_n.
\end{equation*}
If the two terms, $\beta_1$ and $\beta_2$ do not have any common factor then with the same method as  we used in (iii), we can rewrite the binomial $\beta$ by adding and subtracting a term $m$ of degree $n$, which has a common factor $m_1$ with $\beta_1$ and a common factor $m_2$ with $\beta_2$, and we will have

\begin{equation*}
\beta_1+\beta_2= \beta_1+m+\beta_2-m=m_1(c_1+m')+m_2(c_2-m''),
\end{equation*}
where $\beta_1=m_1c_1$, $m=m_1m'=m_2m''$ and $\beta_2=m_2c_2$. Since $c_1+m'$ and $c_2-m''$ are of degree at most $n-1$, by the induction assumption we have
\begin{equation*}
\beta_1+\beta_2=m_1(c_1+m')+m_2(c_2-m'')\in(Q)_n.
\end{equation*}
This completes the induction step and the proof of the proposition.

\end{proof}

\vspace{.2in} 

Recall that for the generic symmetric $n\times n$ matrix $X$, the unacceptable monomials of degree $k$ in $S^s_k$ are the monomials which do not divide any term of the determinant of $Y=\xi(X)$, and recall that we denote the set of degree $k$ unacceptable monomials by $U_k$.

\vspace{.2in}

\begin{lem}\label{lem:det-symm-unacceptable-traiangularity}
We can write each unacceptable monomial of $S^s_k$ ($2\leq k\leq n$), as an explicit element of the product $S^s_{k-2}\cdot \langle Q\rangle $, where $\langle Q\rangle \subset S^s_2$ is the space defined in the Proposition \ref{prop:gen-det-symm}.

\end{lem}

\begin{proof}
We use induction on $k$. For $k=2$ the claim is obviously true.  We show that the claim is true for $k=3$. We need to show that the space $U_3$ of unacceptable monomials in $S^s_3$ are in $S^s_1\langle Q\rangle $. The unacceptable monomials of degree 3 for the $n\times n$ generic symmetric matrix have one of the following forms:

(a) Unacceptable of the form $x^2y$ where $x$ is a diagonal element.

(b) Unacceptable of the form $xyz$ where $x$ is a diagonal element, $y\neq x$ is in the same row or column with $x$ and $z\neq x$. 

(c) Unacceptable of the form $xyz$ where $x,y,z$ are nondiagonal elements from the same row or column (can be equal to each other).  

Unacceptable monomials of type (a) or (b) are multiples of unacceptable monomials of degree 2, so they are in the space $S^s_1U_2$. So we only need to show that the degree 3 unacceptable monomials of type (c) are in $S^s_1\langle Q\rangle$. The 3 nondiagonal elements in the same row or column of the matrix $X$ are from a symmetric $4\times 4$ sub-matrix. So without loss of generality we show that a degree 3 monomial of type (c) from the following sub-matrix is in $S^s_1\langle Q\rangle$. Let $A^s$ be the $4\times 4$ symmetric sub-matrix of a generic symmetric $n\times n$ matrix,$$
A^s=\left(%
\begin{array}{cccc}
 
  a & b & c & d\\
  b & e & f & g\\
  c & f & h & i\\
  d & g & i & j\\

  \end{array}%
\right),$$

and $D^s$ be the matrix

$$
D^s=\left(%
\begin{array}{cccc}
 
  A & B & C & D\\
  B & E & F & G\\
  C & F & H & I\\
  D & G & I & J\\

  \end{array}%
\right)$$

Monomials of type (c) in $S^s_3$ annihilating $\det(A^s)$ can have one of the following forms:

(I) All three non-diagonal variables are distinct. Consider the monomial $\eta_1=BCD$, a degree three unacceptable monomial of type (c). We have $AF+BC\in \langle Q\rangle$ (since it is a permanent with one diagonal element), so we have $(AF+BC)\circ \det(A^s)=0.$ Hence, $$D(AF+BC)\circ \det(A^s)=0.$$ We also know that $DAF\in S^s_1U_2\subset U_3$ so $DAF\circ  \det(A^s)=0.$ We have $\eta_1=BCD=D(AF+BC)(\text{mod }S^s_1U_2)$. So we have $\eta_1\in S^s_1\langle Q\rangle$. \par

(II) There are only two distinct non-diagonal variables. Consider the monomial $\eta_2=B^2C$, also of type (c). We have $AF+BC\in \langle Q\rangle$ (since it is a permanent with one diagonal element), so we have $(AF+BC)\circ \det(A^s)=0.$ Hence, $$B(AF+BC)\circ \det(A^s)=0.$$ We also know that $BAF\in S^s_1U_2\subset U_3$ so $B^2C\circ  \det(A^s)=0.$ We have $\eta_2=B^2C=B(AF+BC)(\text{mod }S^s_1U_2)$. So we have $\eta_2\in S^s_1\langle Q\rangle$. \par

(III) There is only one non-diagonal variable. Consider the monomial $\eta_3=B^3$, also of type (c). We have $B^2+2AE\in \langle Q\rangle$ (since it is a diagonal permanent with the coefficient 2), so we have $(B^2+2AE)\circ \det(A^s)=0.$ Hence, $$B(B^2+2AE)\circ \det(A^s)=0.$$ We also know that $BAE\in S^s_1U_2\subset U_3$ so $B^3\circ  \det(A^s)=0.$ We have $\eta_3=B^3=B(B^2+2AE)(\text{mod }S^s_1U_2)$. So we have $\eta_3\in S^s_1\langle Q\rangle$. \par

So the lemma is proven for $k=3$ and we have $U_3\subset S^s_1\langle Q\rangle$. Let $P$ denote the subspace of $\langle Q\rangle $ generated by binomials of type (b) and (c) defined in Proposition \ref{prop:gen-det-symm}. We have shown that $U_3\subset S^s_1(U+P)$.

Next assume that $k\ge 4$ and the lemma is established for all integers less than $k$. We want to show that the claim is true for $k$. Let $\mu=\mu_1\mu_2 \cdots \mu_k$ be an unacceptable monomial of degree $k$. We can write $\mu$ such that $\mu_2 \cdots \mu_k$ is an unacceptable monomial of degree $k-1$ so we have
$$\mu=\mu_1(\mu_2 \cdots \mu_k)\in S^s_1(S^s_{k-3}\langle Q\rangle)=S^s_{k-2} \langle Q\rangle. $$
So the lemma is true also for $k$.
\end{proof}
\vspace{.2in}

\begin{Notation} We use the following definitions and notations in the remaining part of this section. 

\begin{itemize}
\item By Lexicographic/Conca order we mean the lexicographic term order induced by the variable order,
$$
Y_{1,1}>Y_{1,2}>\cdots>Y_{1,n}>Y_{2,2}>\cdots>Y_{2,n}>\cdots>Y_{n-1,n}>Y_{n,n}.
$$

By Reverse Lexicographic order we mean the lexicographic term order induced by the variable order,
$$
Y_{1,1}<Y_{1,2}<\cdots <Y_{1,n}<Y_{2,2}<\cdots <Y_{2,n}< \cdots <Y_{n-1,n}<Y_{n,n}.
$$

\item Let $\bf M$ be the $k\times k$ minor of the generic symmetric matrix $X$, with the set of rows $\{a_1,\ldots ,a_k\}$ and the set of the columns $\{b_1,\ldots,b_k\}$ where $a_1<\cdots<a_k$ and $b_1<\cdots< b_k$. Then the initial monomial of $\bf M$ using the lexicographic (Conca) order is  $x_{a_1b_1}\cdots x_{a_kb_k}$.

\item We denote by $[a_1,\ldots,a_k|b_1,\ldots,b_k]$ the $k\times k$ doset minor where the sequence of rows  $a=(a_1,\ldots,a_k)$ and the sequence of columns $b=(b_1,\ldots,b_k)$ are each subsequences of $1,\ldots,n$ satisfying the following conditions:
$$a_1<a_2<\cdots<a_k,$$
$$b_1<b_2<\cdots<b_k,$$
$$a_i\leq b_i, \text{ }\forall \text{ }1\leq i\leq k.$$

\item We denote by $(a_1,\ldots,a_k|b_1,\ldots,b_k)$ the acceptable monomial $x_{a_1b_1}\cdots x_{a_kb_k}$. Note that we write the acceptable monomial $m=(a_1,\ldots,a_k|b_1,\ldots,b_k)$, with $a=(a_i)$ an increasing sequence. But unlike the doset minors, the sequence $b=(b_i)$ doesn't need to be increasing.

\item A \emph{Conca monomial} is the initial monomial of a doset minor in lexicographic order. 

\item The set of all $k\times k$ doset minors form a Gr\"obner basis for the ideal generated by all $k\times k$ minors (Theorem \ref{thm:doset-Conca}). Hence the ideal generated by the set of initial monomials of all minors is equal to the ideal generated by the set of the initial monomials of the doset minors.
\item Let $A_k$ be the set of acceptable monomials in $S^s_k$.
\item Let $\iota:R^s \rightarrow S^s,\iota(x_{ij})=y_{ij}$, and $C_k$ be the subset of $A_k$ defined by $$\{\iota(\mu)|\mu \text{ a Conca initial monomial (in lex order) of a } k\times k \text{ doset minor of } X\}.$$
\item Let $C'_k$ be the complementary set to $C_k$ of acceptable monomials in $A_k$.
\item For each $\mu \in A_k$, let $A_{>\mu}$, denote the subset of elements $\nu \in A_k$, such that $\nu>\mu$ in the lexicographic order of $S^s$.
\item For the monomial $m=x_{a_1b_1}\cdots x_{a_kb_k}$ denoted by $(a_1,\ldots,a_k|b_1,\ldots ,b_k)$, we call a pair $(b_i,b_j)$, with $i<j$, a \emph{reversal pair} if $b_i\ge b_j$. As an example, in $(1,2,3|6,4,5)$, $6\ge 4$ so $(6,4)$ is a reversal pair.

\end{itemize}

\end{Notation}
\vspace{.2in}
\begin{prop}\label{prop:triangularity-det} Each acceptable non-Conca monomial of degree $k$ ($3\leq k\leq n$), is the initial monomial (in the reverse lex order) of an element of $S^s_{k-2}\langle Q\rangle$. 
\end{prop}

\begin{proof}
We use induction on $k\ge 3$. First let $k=3$.  Let $\mu$ be a degree 3 acceptable monomial which is not the initial term of any $3\times 3$ doset minor in the lexicographic order. We want to show that $\mu$ is the initial monomial of an element of $S^s_1\langle Q\rangle$ in the reverse lexicographic order. The acceptable monomials $y_{i_1i_2}y_{i_3i_4}y_{i_5i_6}$ of degree 3 for the $n\times n$ generic symmetric matrices can be listed as follows:

(a) All 6 indices are distinct.

(b) There is one repeated index.

(c)  There are 2 repeated indices.

(d) There are 3 repeated indices.
\vspace{0.2in}
We discuss each of the above types separately. In each case we show that $\mu$ is the initial term of an element of $S^s_{1}\langle Q\rangle$ in the reverse lex order.

(a) All 6 indices are distinct $m=y_{i_1i_2}y_{i_3i_4}y_{i_5i_6}$, $(i_1,i_3,i_5|i_2,i_4,i_6)$. Without loss of generality we can assume these indices are 1,2,3,4,5,6. In order to have a non-Conca monomial of this kind, it is enough to have at least one reversal pair. Let $m$ be a monomial with at least one reversal pair. Then $m$ is not the initial monomial of any $3\times 3$ minor of $Y$ in the lexicographic order. So it is not in the ideal generated by all the initial monomials of the $3\times 3$ minors of $Y$. Hence by Theorem  \ref{thm:doset-Conca} it is not in the ideal generated by all the initial monomials of all $3\times 3$ doset minors of $Y$. 

In the doset minor $[i_1,i_3,i_5|i_2,i_4,i_6]$, we have
$$i_1<i_3<i_5,$$
$$i_2<i_4<i_6.$$
Without loss of generality we can assume $i_1=1,i_3=2$, $i_5=3$ $i_2=4,i_4=5$ and $i_6=6$. In this case, the monomial $y_{14}y_{25}y_{36}$ is the initial term in the corresponding $3\times 3$ doset minor using the lexicographic order. 

The corresponding $6\times 6$ symmetric submatrices are 
$$
X=\left(%
\begin{array}{cccccc}
 
  a & b & c & d & e & f \\
  b & g & h & i & j & k  \\
  c & h & l & m & n & o\\
  d & i & m & p & q & r\\
  e & j & n & q & s & t\\
  f & k & o & r & t & u\\
  
   \end{array}%
\right),$$

$$
Y=\left(%
\begin{array}{cccccc}
 
   A & B & C & D & E & F \\
  B & G & H & I & J & K  \\
  C & H & L & M & N & O\\
  D & I & M & P & Q & R\\
  E & J & N & Q & S & T\\
  F & K & O & R & T & U\\

   \end{array}%
\right).$$

Consider a non-Conca degree three monomial involving 6 distinct rows and columns. Each non-initial Conca monomial has at least one reversal pair $(i_j,i_k),(j<k)$ where $j,k\in\{2,4,6\}$ such that $i_j\ge i_k$. Without loss of generality we consider the monomial $\mu=FIN=(1,2,5|6,4,3)$. The hafnian of the following $4\times 4$ symmetric sub-matrix with the rows and columns 1,2,4,6 is an element of $\langle Q\rangle $ by Proposition \ref{prop:gen-det-symm} 

$$
\mathrm{Haf} \left(%
\begin{array}{cccc}
 
   A & B & D & F \\
  B & G & I & K  \\
  D & I &  P & R\\
  F & K & R & U\\

   \end{array}%
\right)=BR+DK+FI.$$

Hence for $\mu = FIN,$ we have $f_\mu=N(BR+DK+FI)\in S_1^s\langle Q\rangle$, where $N(BR+DK)\in A_{\mu}$. So $\mu$ is the initial term of $N$ times a hafnian in in the reverse lex order.

(b) There is one repeated index. Without loss of generality we can assume these indexes are 1,2,3,4,5, with one of them repeated. In order to have a non-Conca example of this kind, it is enough to have 1 reversal pair. For example in $(1,2,3|4,1,5)$, $4>1$ is a reversal pair. We form the $5\times 5$ symmetric matrix with these rows and columns, here

$$
X=\left(%
\begin{array}{ccccc}
 
  a & b & c & d & e \\
  b & f & g & h & i \\
  c & g & j & k & l \\
  d & h & k & m & n\\
  e & i & l & n & o \\

   \end{array}%
\right),$$

$$
Y=\left(%
\begin{array}{ccccc}
 
  A & B & C & D & E \\
  B & F & G & H & I \\
  C & G & J & K & L \\
  D & H & K & M & N\\
  E & I & L & N & O \\

   \end{array}%
\right).$$

The monomial $\mu=y_{14}y_{21}y_{35}=BDL$ as an acceptable monomial of type (b). Consider the hafnian of the following $4\times 4$ symmetric sub-matrix with the rows and columns 1,2,3,5,

$$
\mathrm{Haf} \left(%
\begin{array}{cccc}
 
   A & B & C & E \\
  B & F & G & I  \\
  C & G &  J & L\\
  E & I & L & O\\

   \end{array}%
\right)=BL+CI+EG.$$

Given $\mu=BDL$, we have $f_{\mu}=D(BL+CI+EG)\in S^s_1\langle Q\rangle, \text{ where } D(CI+EG)\in A_{>\mu}.$ So $\mu=BDL$ is the initial term of $D$ times a hafnian in in the reverse lex order.

(c) There are 2 repeated indices, Without loss of generality we can assume these indexes are 1,2,3,4, with two of them repeated. In order to have a non-Conca example of this kind, it is enough to have one reversal pair. For example in $(1,2,3|2,1,4)$, $2> 1$. We can form a $4\times 4$ symmetric matrix with these rows and columns,

$$
X=\left(%
\begin{array}{cccc}
 
  a & b & c & d  \\
  b & e & f & g \\
  c & f & h & i \\
  d & g & i & j \\

   \end{array}%
\right),$$

$$
Y=\left(%
\begin{array}{cccc}
 
  A & B & C & D  \\
  B & E & F & G \\
  C & F & H & I \\
  D & G & I & J \\

   \end{array}%
\right).$$

Now we consider the monomial $\mu=y_{12}y_{21}y_{34}=B^2I$ as an acceptable monomial of type~(c). Given $\mu=B^2I$, we have $f_{\mu}=I(B^2+2AE)\in S^s_1\langle Q\rangle$ since $B^2+2AE$ is a binomial in $\langle Q\rangle $ (see Proposition \ref{prop:gen-det-symm}). We have $AEI \in A_{>\mu}$ so $\mu=B^2I$ is the initial term of $I$ times a binomial in $\langle Q\rangle $ in in the reverse lex order.

(d) There are 3 repeated indices, Without loss of generality we can assume these indexes are 1,2,3, all of them repeated. In order to have a non-Conca example of this kind, it is enough to have 1 reversal pair. For example in $(1,2,3|3,2,1)$, $3> 1$. We can form a $3\times 3$ symmetric matrix with these rows and columns,

$$
X=\left(%
\begin{array}{ccc}
 
  a & b & c  \\
  b & d & e  \\
  c & e & f \\

   \end{array}%
\right),$$

$$
Y=\left(%
\begin{array}{ccc}
 
  A & B & C  \\
  B & D & E  \\
  C & E & F \\

   \end{array}%
\right).$$

Now we consider the monomial $\mu=y_{13}y_{22}y_{31}=C^2D$ as an acceptable monomial of type (c). Given $\mu=C^2D$, $f_{\mu}=C(BE+CD)\in S^s_1\langle Q\rangle, \text{ where } BEC \in A_{>\mu}.$\par
\vspace{0.2in}
Since all other cases are similar to the above examples, for $k=3$ the claim of Proposition \ref{prop:triangularity-det} is true. Now assume that the Proposition \ref{prop:triangularity-det} is true for all integers less than $k$. We have to show that the Proposition \ref{prop:triangularity-det} is also true for $k$. Let $\mu=y_{i_1j_1}\cdots y_{i_kj_k}=(i_1,\ldots,i_k|j_1,\ldots,j_k)$ be a degree $k$ acceptable non-Conca monomial, so it has at least one reversal pair. We can consider $\mu$ as the product of one variable, $y_{ab}$ and a degree $k-1$ acceptable non-Conca monomial, $\mu_1$, containing at least one reversal pair. Then by the induction assumption $\mu_1$  is the initial monomial (in rev. lex.) of an element of $S^s_{k-3}\langle Q\rangle$. So we have $\mu=y_{ab}\mu_1$ is the initial monomial (in rev. lex.) of an element of $S^s_{k-2}\langle Q\rangle$. This completes the proof.
\end{proof}
\vspace{.2in}

\begin{ex}
Consider the case $n=3$. we have 
$$
X=\left(%
\begin{array}{ccc}
 
  a & b & c \\
  b & d & e \\
  c & e & f \\
  
   \end{array}%
\right),$$

$$
Y=\left(%
\begin{array}{ccc}
 
  A & B & C \\
  B & D & E \\
  C & E & F \\
  
   \end{array}%
\right),$$

There are 5 acceptable degree 3 monomials. Using the lexicographic term order induced by the variable order,
$$
Y_{1,1}>Y_{1,2}>Y_{1,3}>Y_{2,2}>Y_{2,3}>Y_{3,3},
$$

we have the following order on the degree 3 acceptable monomials
$$
ADF>AE^2>B^2F>BEC>C^2D.
$$

\begin{itemize}

\item The set of Conca initial monomials of degree three, $C_3$, is the subspace spanned by the set $\{ADF\}$.
\item The set of all acceptable degree three monomials that are not in $C_3$ is spanned by $$C'_3=\{AE^2,B^2F,BEC,C^2D\}.$$
\item For $\mu_1=C^2D$, $f_{\mu_1}=C^2D+2ADF=D(C^2+2AF)\in S^s_1\langle Q\rangle, \text{ where } ADF\in A_{>\mu_1}.$
\item For $\mu_2=BCE$, $f_{\mu_2}=BEC+AE^2=E(BC+AE)\in S^s_1\langle Q\rangle,  \text{ where } AE^2\in A_{>\mu_2}.$
\item For $\mu_3=B^2F$, $f_{\mu_3}=B^2F+2ADF=F(B^2+2AD)\in S^s_1\langle Q\rangle , \text{ where } ADF\in A_{>\mu_3}.$
\item For $\mu_4=AE^2$, $f_{\mu_4}=A(E^2+2DF)\in S^s_1\langle Q\rangle, \text{ where } ADF\in A_{>\mu_4}.$

\end{itemize}
\vspace{0.2in}

Hence each acceptable non-Conca monomial of degree three is the initial monomial (in the reverse Lex. order) of an element of $S^s_1\langle Q\rangle$. 

\end{ex}

\vspace{.2in}
\begin{cor}\label{cor:main-symm-det-k}
For $1\leq k \leq n$ we have
\begin{equation*}
(Q)_k=\mathrm{Ann}(\det(X))\cap S^{s}_{k}.
\end{equation*}

We also have $(Q)_{n+1}=S^{s}_{n+1}$.

\end{cor}

\begin{proof}
By Lemmas \ref{lem:easydirectionk-det-dymm} and \ref{lem:Scircledet-symm} we have

$$
S^s_{k-2}\langle Q\rangle=(Q)_k\subset \mathrm{Ann}(\det(X))\cap S^s_k.
$$

By Remark \ref{remark:introIK} and Lemma \ref{lem:Scircledet-symm} we have
\begin{equation*}
(\mathrm{Ann}(\det(X)))_k=(\mathrm{Ann}(S^s_{n-k}\circ(\det(X)))_k=(\mathrm{Ann}(M_{k}(X)))_k
\end{equation*}

So we have

\begin{equation*}
\dim S^s_{k-2}\langle Q\rangle\leq \dim(\mathrm{Ann}(\det(X))\cap S^s_k)=\dim S^s_k-\dim M_k(X).
\end{equation*}

On the other hand, by definition the sets $U_k$ and $C'_k$ are linearly independent and form a basis for the corresponding subspaces. Hence by Lemma \ref{lem:det-symm-unacceptable-traiangularity} and Proposition \ref{prop:triangularity-det} we have

$$
\dim S^s_k-\dim M_k(X)=\dim \langle C'_k \rangle+\dim \langle U_k \rangle  \leq \dim S^s_{k-2}\langle Q\rangle.
$$

So we have
$$
\dim (Q)_k=\dim S^s_{k-2}V=\dim S^s_k-\dim M_k(X)=\dim( \mathrm{Ann}(\det(X))\cap S^s_k).
$$

\end{proof}

\vspace{.2in}

\begin{thm}\label{main-theorem-det-symmetric}
Let $X$ be a generic symmetric $n\times n$ matrix. Then the apolar ideal $\mathrm{Ann}(\det(X))$ is the ideal $(Q)$ and is generated in degree 2.
\end{thm}

\begin{proof}
This follows directly from Lemma \ref{lem:det-symm-unacceptable-traiangularity}, Proposition \ref{prop:triangularity-det} and Corollary \ref{cor:main-symm-det-k}.
\end{proof}
\vspace{.2in}

\begin{prop}\label{prop:grobner-det-symmetric}
The set $Q $ is a Gr\"{o}bner basis for the ideal  $\mathrm{Ann}(\det(X))$.
\end{prop}

\begin{proof}
We have shown that $Q$ generates  $\mathrm{Ann}(\det(X))$, and we use Buchberger's Algorithm to show that $Q$ is a Gr\"{o}bner basis for the ideal $\mathrm{Ann}(\det(X))$.

(1) Let $\mathcal F$ and $\mathcal G$ and be two distinct permanents of $Y$ of type (c) in Proposition \ref{prop:gen-det-symm}. Let $\mathcal F=y_{ii}y_{jk}+y_{ik}y_{ji}$ and $\mathcal G=y_{uu}y_{zv}+y_{uv}y_{zu}$.

\begin{center}  $ \mathcal F= \mathrm{perm}\left(%
\begin{array}{cc}
 
  y_{ii} & y_{ik} \\
  y_{ji} & y_{jk}\\
  \end{array}%
\right)$.

\end{center}

\begin{center}  $\mathcal G= \mathrm{perm}\left(%
\begin{array}{cc}
 
  y_{uu} & y_{uv} \\
  y_{zu} & y_{zv}\\
  \end{array}%
\right)$.

\end{center}

Let $f_1=y_{ii}y_{jk}$ be the leading term of $\mathcal F$, and $g_1=y_{uu}y_{zv}$ be the leading term of $\mathcal G$ with respect to Conca monomial order. Denote the least common multiple of $f_1$ and $g_1$ by $h$. Then we have:

$$ 
S(\mathcal F,\mathcal G)=(h/f_1)F-(h/g_1)G=y_{uu}y_{zv}y_{ik}y_{ji}-y_{ii}y_{jk}y_{uv}y_{zu}.
$$

Now using the multivariate division algorithm, we reduce $S(\mathcal F,\mathcal G)$ relative to the set $Q$. When there is no common factor in the initial terms of $\mathcal F$ and $\mathcal G$ the reduction is zero. First we reduce $S(\mathcal F,\mathcal G)$ dividing by $\mathcal F$, so we will have

$$
S(\mathcal F,\mathcal G)+y_{uv}y_{zu}\mathcal F=y_{uu}y_{zv}y_{ik}y_{ji}+y_{uv}y_{zu}y_{ik}y_{ji}.
$$

Then we reduce the result using $\mathcal G$ this time, so we will have
$$
y_{uu}y_{zv}y_{ik}y_{ji}+y_{uv}y_{zu}y_{ik}y_{ji}-y_{ik}y_{ji}\mathcal G=0.
$$

So we have shown that for all pairs $\mathcal F$ and $\mathcal G$ of distinct permanents of $Y$ of type (c), the $S$-polynomials $S(\mathcal F,\mathcal G)$ reduce to zero with respect to $ Q$.

(2) Let $\mathcal F=y_{ii}y_{jk}+y_{ik}y_{ji}$ and $\mathcal G=y_{ii}y_{lm}+y_{im}y_{li}$ be two permanents whose initial terms have a common factor. We have

$$
S(\mathcal F,\mathcal G)=y_{lm}y_{ik}y_{ji}-y_{jk}y_{im}y_{li}.
$$

Without loss of generality we can restrict to a $5\times 5$ symmetric submatrix. Note that in a $5\times 5$ symmetric sub-matrix we can have two hafnians whose initial terms have one common factor, two permanents whose initial terms have a common factor, and a permanent and a hafnian whose initial terms have a common factor. Denote the $5\times 5$ symmetric submatrix by

\begin{center}  $ \left(%
\begin{array}{ccccc}
 
  A & B & C & D & E \\
  B & F & G & H & I \\
  C & G & J & K & L \\
  D & H & K & M & N \\
  E & I & L & N & O \\
  \end{array}%
\right)$.

\end{center}

Without loss of generality we consider the two permanents $\mathcal F=AG+BC$ and $\mathcal G=AN+DE$.
$$
S(\mathcal F,\mathcal G)= BCN-DEG.
$$
We checked using the multivariate division algorithm in Macaulay 2 that the binomial $BCN-DEG$ reduces to zero mod the initial set of generators $Q$.

Note that any two $4\times 4$ hafnians with the same initial term are exactly the same. So for the hafnians it is enough to consider the $S$-polynomials of hafnians whose initial terms have only one common factor, and of the hafnians whose initial terms do not have a common factor. We should also consider the $S$-polynomials in the case that we have a hafnian and a permanent.

(3) Let $\mathcal F$ and $\mathcal G$ and be two distinct hafnians of $Y$ whose initial terms do not have a common factor. Without loss of generality we can restrict to a $5\times 5$ symmetric matrix as in (2), and consider the two hafnians $\mathcal F=HL+IK+GN$ and $\mathcal G=DG+CH+BK$.

$$
S(\mathcal F,\mathcal G)= BKHL+BIK^2-DG^2N-CGHN.
$$

The multivariate division algorithm in Macaulay 2 shows that the $S$-polynomial $BKHL+BIK^2-DG^2N-CGHN$ reduces to zero.

(4) Let  Let $\mathcal F$ and $\mathcal G$ and be two distinct hafnians of $Y$ whose initial terms have a common factor. Without loss of generality we can restrict to a $5\times 5$ symmetric matrix as in (2). 

Without loss of generality we consider the two hafnians $\mathcal F=CH+DG+BK$ and $\mathcal G=CI+EG+BL$.

$$
S(\mathcal F,\mathcal G)= CHL+DGL-CIK-EGK.
$$

Using the multivariate division algorithm in Macaulay 2, it is easy to see that the polynomial, $CHL+DGL-CIK-EGK$, reduces  to zero. We also show the reduction process for this example directly. We want to reduce the polynomial $S(\mathcal F,\mathcal G)$ using the set $Q$. The initial term for this polynomial is $CHL$. So we should find elements of the set $Q$ other than $\mathcal F$ and $\mathcal G$ whose initial terms divide $CHL$. We have the following three possibilities:

(a) The initial term is $CH$. There is no permanent or hafnian with this initial term in the set $Q$.

(b) The initial term is $CL$. The only element of the set $Q$ with this initial term is the permanent $CL+EJ$.

(c) The initial term is $HL$. There is no permanent or hafnian with this initial term in the set $Q$.

So we reduce $S(\mathcal F,\mathcal G)$ using $CL+EJ$, and we get

$$S'=S(\mathcal F,\mathcal G)-H(CL+EJ)=-CIK+DGL-EGK-JHE.$$

Now the initial term of $S'$ is $CIK$, and we again do the reduction process. Here we have three different possibilities to choose an element from $Q$.

(a') The initial term is $CI$. There is no permanent or hafnian with this initial term in the set $Q$.

(b') The initial term is $CK$. The only element of the set $Q$ with this initial term is the permanent $CK+DJ$.

(c') The initial term is $IK$. There is no permanent or hafnian with this initial term in the set $Q$.

So we reduce $S'$ using $CK+DJ$, and we get

$$S''=S'-I(CK+DJ)=DGL-EGK-JHE+DIJ.$$

Again we look at the three different degree 2 monomials which divide the initial term of $S''$, we have

(a'') The initial term is $DG$. There is no permanent or hafnian with this initial term in the set $Q$.

(b'') The initial term is $GL$. The only element of the set $Q$ with this initial term is the permanent $GL+IJ$.

(c'') The initial term is $DL$. There is no permanent or hafnian with this initial term in the set $Q$.

So we reduce $S''$ using $GL+IJ$, and we get

$$S'''=S''-D(GL+IJ)=-E(GK+HJ)\in V.$$

So the $S$-polynomial can be reduced to zero using the set $\langle Q\rangle $.

(5) Let $\mathcal F$ be a permanent and $\mathcal G$ be a hafnian of $Y$ whose initial terms do not have a common factor. Without loss of generality we can restrict to a $5\times 5$ symmetric matrix as in (2). 

Without loss of generality we consider two permanents $\mathcal F=2FJ+G^2$ and $\mathcal G=CI+EG+BL$.

$$
S(\mathcal F,\mathcal G)=BLG^2-2FJCI-2FJEG.
$$

The multivariate division algorithm in Macaulay 2 shows that the $S$-polynomial, $BLG^2-2FJCI-2FJEG$ can be reduced to zero using the set $\langle Q\rangle $.

(6) Let  Let $\mathcal F$ be a permanent and $\mathcal G$ and be a hafnian of $Y$ whose initial terms have a common factor. Without loss of generality we can restrict to a $5\times 5$ symmetric matrix as in (2), 
and consider two permanents $\mathcal F=BG+CF$ and $\mathcal G=CI+EG+BL$.

$$
S(\mathcal F,\mathcal G)=CFL-CIG-EG^2.
$$

The multivariate division algorithm in Macaulay 2 shows that the $S$-polynomial,  $CFL-CIG-EG^2$  reduces to zero.

\end{proof}


\subsection{Apolar ideal of the permanent} 
\label{subsection:Apolar ideal of the permanent-symm}

In this section we determine the apolar ideal of the permanent of the $n\times n$ generic symmetric matrix, and we will show that it is generated by degree two and degree three polynomials. We first determine the generators of degree two (Proposition \ref{prop:gens-deg2-perm-symm}). We then determine the degree three generators, which occur when $n\ge 6$ (Lemma  \ref{prop:gens-deg3-perm-symm}). A key step is to use triangularity to show that these degree two and degree three elements, generate the apolar ideal (Lemma \ref{lem:perm-sym-unacceptable-trangularity} and Proposition \ref{prop:main-triangularity-perm-symm}).  This leads to our main result (Theorem 3.23).

\vspace{.2in}

Analogous to Proposition \ref{prop:gen-det-symm} we have:

\vspace{.2in}

\begin{prop}\label{prop:gens-deg2-perm-symm}
For an $n\times n$ symmetric matrix $X=(x_{ij})$, the annihilator $\mathrm{Ann}(\mathrm{Perm}(X))\subset S^s$, includes the following degree 2 polynomials:

(a) Unacceptable monomials including $y_{ii}y_{ij}$ for all $1\leq i,j \leq n$ ($i=j$ is allowed). The number of these monomials is $n^2$.

(b) All the $2\times 2$ diagonal minors of $Y$ with a coefficient  $2$ on the diagonal term, i.e. $y_{ij}^2-2y_{ii}y_{jj}$ ($i\neq j$). The number of these binomials is $n\choose 2$.

(c) All the $2\times 2$ minors of $Y$ with one diagonal element, i.e. $y_{ji}y_{il}-y_{jl}y_{ii}$ ($i$, $j$, and $l$ are pairwise distinct). The number of these binomials is $n\cdot{{n-1}\choose 2}$.

\end{prop}

\begin{proof}
We have $\mathrm{Perm}(X)=\sum _ {\sigma \in S^s_n}  \Pi x_{i,\sigma(i)}$. First we show that the monomials of type (a) are in $\mathrm {Ann}(\mathrm{Perm}(X))$. By symmetry we have

$$y_{ii}y_{ij}\circ \mathrm{Perm}(X)=0 (\text{where } j\geq i),$$ $$y_{ii}y_{ji}\circ\mathrm{Perm}(X)=0 (\text{where } j\leq i).$$

Next we want to show that the binomials of type (b) are in $\mathrm {Ann}(\mathrm{Perm}(X))$. Let $M_{\widehat{\{i,j\},\{i,j\}}}$ be an $(n-2)\times(n-2)$ submatrix of $X$, which does not include the rows and columns $i$ and $j$. Let $M=2y_{ii}y_{jj}-y_{ij}^2$. 

$$y_{ii}y_{jj}\circ \mathrm{Perm}(X)=Perm(M_{\widehat{\{i,j\},\{i,j\}}}),$$
$$y_{ij}^2\circ \mathrm{Perm}(X)=2Perm (M_{\widehat{\{i,j\},\{i,j\}}}).$$

Hence we have $M\circ \mathrm{Perm}(X)=0$.

To show that the binomials of type (c) are in the annihilator ideal we can use a similar proof to that used for the binomials of type (b). Let $M_{\widehat{\{i,j\},\{i,l\}}}$ be an $(n-2)\times(n-2)$ submatrix of $X$, which does not include the rows $i$ and $j$ and the columns $i$ and $l$.

$$y_{ji}y_{il}\circ \mathrm{Perm}(X)=2\mathrm{Perm}(M_{\widehat{\{i,j\},\{i,l\}}}),$$
$$y_{jl}y_{ii}\circ \mathrm{Perm}(X)=2\mathrm{Perm}(M_{\widehat{\{i,j\},\{i,l\}}}).$$

Hence we have $M\circ \mathrm{Perm}(X)=0$.

\end{proof}

\vspace{.2in}

\begin{defi}\label{def:Wperm-symm}
We denote by $Q'$ the set of the degree 2 elements of type (a), (b) and (c) in Proposition \ref{prop:gens-deg2-perm-symm}, and by $\langle Q' \rangle$ the vector subspace of $S^s$ spanned by $Q'$. We denote by $Q'_a$, $Q'_b$,and  $Q'_c$ the set of elements in $(a)$, $(b)$,and $(c)$ respectively.
\end{defi}
\vspace{.2in}

Analogous to Lemma \ref{lem:symm-det-dimV} we have:
\vspace{.2in}

\begin{lem}\label{lem:symm-perm-dimw}
The set $Q'$ is linearly independent and we have,
  \begin{equation*}
 \dim_{\sf k}\langle Q' \rangle =n^2+{n\choose 2}+n\cdot{{n-1}\choose 2}.
 \end{equation*}
 
 \end{lem}
 
\begin{proof}
Each of the three subsets are linearly independent from each other since they involve different variables. So if we show that each subset is linearly independent we are done. The subset $Q'_a$ is linearly independent since the monomials in $Q'_a$ form a Gr\"{o}bner basis for the ideal they generate. The subsets $Q'_b$ and $Q'_c$ are linearly independent since by choosing two elements of the matrix, where at least one element is diagonal, we have a unique $2\times 2$ minor. Hence the set $Q'$ is linearly independent and the dimension of the vector space $\langle Q' \rangle$ is $n^2+{n\choose 2}+n\cdot{{n-1}\choose 2}$.

 \end{proof}
 
 \begin{Notation}
For a generic symmetric $n\times n$ matrix $X$, we denote by $P_k(X)$ the space of the permanents of all $k\times k$ sumatrices of $X$.

\end{Notation}

\vspace{.2in}
Analogous to Lemma \ref{lem:Scircledet-symm} we have:
\vspace{.2in}
\begin{lem}\label{lem:Scircleperm-symm} Let $1\leq k\leq n$. We have
 \begin{equation*}
S^s_{k}\circ(\mathrm{perm}(X))=P_{n-k}(X) \subset R^s.
 \end{equation*}
 \end{lem}

 \begin{proof}
To show the inclusion
  \begin{equation*}
S^s_{k}\circ(\mathrm{perm}(X))\subset P_{n-k}(X) \subset R^s,
 \end{equation*}

 we use induction on $k$. Let $P_{ij}$ denote the permanent of the submatrix obtained by deleting the $i$-th row and $j$-th column. For $k=1$, we have two different cases:
 
 I) for a diagonal element $y_{ii}$ we have 
 
 $$y_{ii}\circ (\mathrm{perm}(X))=P_{ii} \in P_{n-1}(X).$$
 
 II) Let $y_{ij}$ be a non-diagonal element. Without loss of generality we can consider $y_{12}$. We have $y_{12}=y_{21}$. The monomial $y_{12}^2$ appears in exactly $(n-2)!$ terms coming from $y_{12}^2\cdot (P_{12})_{21}$.

 We also have $2((n-1)!-(n-2)!)$ terms in the permanent which contain $y_{12}$ but do not contain $y_{12}^2$. These terms come from the sub-permanent obtained by deleting the first or second row.
 
 So we have
 
 $$y_{ij}\circ (\mathrm{perm}(X))=2P_{ij} \in P_{n-1}(X).$$

Assume that the above inclusion holds for $k-1$, i.e $$S^s_{k-1}\circ (\mathrm{perm}(X)) \subset P_{n-(k-1)}(X),$$
  and we want to show that it is true for $k$. We have $$S^s_k\circ \mathrm{perm}(X)=S^s_1S^s_{k-1}\circ \mathrm{perm}(X)\subset S^s_1\circ P_{n-k+1}(X)\subset P_{n-k}(X),$$
 
as required. Next we want to show  the opposite inclusion
 \begin{equation*}
S^s_{k}\circ(\mathrm{perm}(X))\supset P_{n-k}(X) \subset R^s,
 \end{equation*}

Let $P_{\widehat I,\widehat J}(X), I=\{i_1,\ldots ,i_k\}, J=\{j_1,\ldots ,j_k\}, 1\le i_1\le i_2\le \cdots \le i_k\le n, 1\le j_1\le j_2\le \cdots \le j_k\le n$ be the $(n-k)\times (n-k)$ permanent of the submatrix of $X$ one obtains by deleting the $I$ rows and $J$ columns of $X$. Let $$\Delta_{(I,J)}=\{(i_r,j_r)|i_r\in I, j_r\in J \text{ and } i_r=j_r\}.$$  Let $\Delta_I=\{i_r|(i_r,j_r)\in \Delta_{(I,J)}\}$ and $\Delta_J=\{j_r|(i_r,j_r)\in \Delta_{(I,J)}\}$. Let $P_{(I,J)-\Delta}$ be the sub matrix of Y with the rows $I-\Delta_I$, and the columns $J-\Delta_J$.

\vspace{0.1in}

\underline{Claim:}

$$P_{\widehat{I},\widehat{J}}= c \prod_{(i_r,j_r)\in\Delta_{(I,J)}}y_{i_rj_r} \mathrm{perm}(P_{(I,J)-\Delta})  \circ (\mathrm{perm}(X))$$ where $c\not= 0\in \sf{k}.$

To prove the claim we use induction on $|I|=|J|=k$, the cardinality of the sets $I$ and $J$.  First we show the claim is true for $k=1$. Let $I=\{i_1\}$ and $J=\{j_1\}$. We have two cases

I. $i_1=j_1$ so $y_{i_1j_1}$ is a diagonal element and we have

\begin{equation*}
P_{\widehat{I},\widehat{J}}=y_{i_1j_1}\circ (\mathrm{perm}(X)).
\end{equation*}

II. $i_1\neq j_1$ so we have

\begin{equation*}
y_{i_1j_1}\circ (\mathrm{perm}(X))=2P_{\widehat{I},\widehat{J}}.
\end{equation*}

So for $k=1$ the claim holds. Assume that the claim holds for every $I$ and $J$ with $|I|=|J|=k-1$ and we want to show that the claim is also true for $I$ and $J$ with $|I|=|J|=k$.

Let $I=\{i_1,\ldots ,i_k\}$ and $J=\{j_1,\ldots ,j_k\}$. 

Let $I'=I-\{i_1\}$ and $J'=J-\{j_1\}$. We have $|I'|-|J'|=k-1$ so by the induction assumption we have 

\begin{equation*}
P_{\widehat{I'},\widehat{J'}}= c \prod_{(i_r,j_r)\in\Delta_{(I',J')}}y_{i_rj_r} \mathrm{perm} (P_{(I',J')-\Delta})  \circ (\mathrm{perm}(X))
\end{equation*}

Writing the Laplace expansion of the permanent using row $i_1$ or column $j_1$ for $P_{\widehat{I},\widehat{J}}$, we get 

$$P_{\widehat{I},\widehat{J}}= c \prod_{(i_r,j_r)\in\Delta_{(I,J)}}y_{i_rj_r} \mathrm{perm} (P_{(I,J)-\Delta})  \circ (\mathrm{perm}(X)),$$ where $c\not=0 \in \sf{k}$. Hence $P_{\widehat{I},\widehat{J}}\in S^s_{n-k} \circ (\mathrm{perm}(X))$.

 \end{proof}
 
 \begin{lem}\label{lem:symm-perm-dimP2}  Let $X$ be a generic symmetric matrix. We have
 
 $$H(S^s/\mathrm{Ann}(\mathrm{Perm}(X)))_{\sf{k}}=\frac{{n\choose k} ({n\choose k}+1)}{2}.$$
 
 So the length $\dim S^s/\mathrm{Ann}(\mathrm{Perm}(X))$ satisfies the following equation
$$\dim S^s/\mathrm{Ann}(\mathrm{Perm}(X))=\frac{{2n\choose n}+2^n}{2}.$$
\end{lem}
 
 \begin{proof}
 Let $P_k$ denote the space of $k\times k$ permanents of the $n\times n$ generic symmetric matrix $X$.  Using Lemma \ref{lem:Scircleperm-symm} we have 
  
 $$H(S^s/\mathrm{Ann}(\mathrm{Perm}(X)))_{\sf{k}}=\dim_{\sf k} P_k={{n\choose k}\choose 2}=\frac{{n\choose k} ({n\choose k}+1)}{2}.$$
 
 Hence we have
 \begin{equation}\label{dim-symm-perm-hilbert}
 \dim S^s/\mathrm{Ann}(\mathrm{Perm}(X))=\sum_{k=0}^{k=n} \frac{{n\choose k} ({n\choose k}+1)}{2}=\frac{{2n\choose n}+2^n}{2}.
 \end{equation}
 
 A combinatorial proof of the second equality in the Equation \ref{dim-symm-perm-hilbert} can be found in \cite{ST}, Example 1.1.17.
 
 \end{proof}
 
Table \ref{table-symm-hilb-perm} gives the Hilbert functions $H(S^s/\mathrm{Ann}(\mathrm{Perm}(X))$ in Equation \ref{dim-symm-perm-hilbert} for a generic $n\times n$ symmetric matrix $X$, for $n\le 8$. The entries in Table \ref{table-symm-hilb-perm} are at least as big as the corresponding entries of Table \ref{table:determinant-symmetric-hilbert}.
 
 \begin{table}[h]
\begin{center}
\caption{The Hilbert sequence of the Permanent of the generic symmetric matrix}\label{table-symm-hilb-perm}

\begin{tabular}{l*{9}{c}r}
\hline
n=2 & 1 & 3 & 1  \\
\hline
n=3           & 1 & 6 & 6 &  1 \\
\hline
n=4 & 1 & 10 & 21 & 10 & 1 \\
\hline
n=5 & 1 & 15 & 55 & 55 & 15 & 1 \\
\hline
n=6 & 1 & 21 & 120 & 210 & 120 & 21 & 1\\
\hline
n=7 & 1 & 28 & 231 & 630 & 630 & 231 & 28 & 1\\
\hline
n=8 & 1 & 36 & 406 & 1596 & 2485 & 1596 & 406 & 36 & 1\\

\hline

\end{tabular}
\end{center}
\end{table}

Analogous to Lemma \ref{lem:det-symm-V-degree2} we have
\vspace{0.2in}
\begin{lem}\label{lem:perm-symm-W-degree2} For the generic symmetric $n\times n$ matrix $X$, we have
\begin{equation*}
\langle Q' \rangle=\mathrm{Ann}(P_2)\cap S^s_2=(\mathrm{Ann}(\mathrm{perm} (X)))_2.
\end{equation*}

\end{lem}

\begin{proof}
By the Lemma \ref{lem:Scircleperm-symm} we have 

\begin{equation*}
\mathrm{Ann}(S^s_{n-2}\circ(\mathrm{perm}(X)))=\mathrm{Ann}(P_2(X)).
\end{equation*}

Remember $\langle Q' \rangle$ is the span of $Q'$, a set of degree two polynomials (Definition \ref{def:Wperm-symm}). By the Proposition \ref{prop:gens-deg2-perm-symm} we have

\begin{equation*}
(\mathrm{Ann}(\mathrm{perm}(X)))_2\supset \langle Q' \rangle
\end{equation*}

By the Remark \ref{remark:introIK} we have

\begin{equation*}
(\mathrm{Ann}(\mathrm{perm}(X)))_2=(\mathrm{Ann}(S^s_{n-2}\circ(\mathrm{perm}(X))))_2\subset \mathrm{Ann}(P_2(X))
\end{equation*}

Hence we have

\begin{equation*}
\langle Q' \rangle \subset \mathrm{Ann}(P_2).
\end{equation*}

On the other hand, using Lemma \ref{lem:symm-perm-dimP2} we have 

\begin{equation*}
\dim \langle Q' \rangle=n^2+{n\choose 2}+n\cdot{{n-1}\choose 2}={{\frac{n^2+n}{2}}+1\choose 2}-\frac{{n\choose 2}({n\choose 2}+1)}{2}=\dim S^s_2-\dim P_2(X).
\end{equation*}

So we have the equality

\begin{equation*}
\langle Q' \rangle=\mathrm{Ann}(P_2)\cap S^s_2.
\end{equation*}

By the Proposition \ref{prop:gens-deg2-perm-symm}, Lemmas \ref{lem:symm-perm-dimw} and \ref{lem:Scircleperm-symm} we have
\begin{equation*}
(\mathrm{Ann}(\mathrm{perm}(X)))_2\subset \mathrm{Ann}(P_2(X))\subset \langle Q' \rangle.
\end{equation*}

So we have

\begin{equation*}
\langle Q' \rangle=(\mathrm{Ann}(\mathrm{perm}(X)))_2.
\end{equation*}

\end{proof}

The apolar ideal of the permanent of the generic symmetric matrix is not generated in degree two in general. For $n=2,3,4,5$ the apolar ideal is generated in degree 2 with the generators $\langle Q' \rangle$ introduced in Proposition \ref{prop:gens-deg2-perm-symm}. We will show that, starting from $n=6$ there are generators of degree 3 in the annihilator ideal (Lemma \ref{prop:gens-deg3-perm-symm}). Here, for the readers' convenience we summarize the information/observations we have about these examples:

\begin{itemize}
\item For $n=2$ we have $$
X=\left(%
\begin{array}{cc}
 
  a & b \\
  b & c \\
  
 \end{array}%
\right).$$

We have $\mathrm{perm}(X)=b^2+ac$. The apolar ideal $I=(C^2,BC,B^2-2AC,AB ,A^2)$. The corresponding Hilbert sequence is $H=(1,3,1)$

\item For $n=2,3,4,5$ the apolar ideal is generated by $Q'$ in degree 2.

\vspace{0.2in}

\item We show that for $n=6,7,8$ the apolar ideal has some degree 3 generators:

The terms $y_{i_1i_2}y_{i_3i_4}y_{i_5i_6}$ that appear in the degree 3 polynomials of the apolar ideal for $n=6,7,8$ follow these rules:

$$i_1\leq i_3 \leq i_5,$$

$$i_1\leq i_2, i_3\leq i_4, i_5\leq i_6.$$

So the terms that appear in the degree three generators of the apolar ideals are exactly the terms that appear in the hafnians of the $6\times 6$ symmetric sub-matrices.

\item For each $6\times 6$ symmetric submatrix $X$, we have five degree three homogeneous polynomials among the generators of $\mathrm{Ann}(\mathrm{perm}(X))$. These five polynomials are linearly independent since they involve different variables. Three of them have six monomial terms and two of them have eight monomial terms. So the number of degree three generators of $\mathrm{Ann}(\mathrm{perm}(X))$ is equal to $5\cdot{n\choose 6}$. We write these five degree-three forms
below for $n=6$ in Lemma \ref{prop:gens-deg3-perm-symm} that follows.
 \end{itemize}
\vspace{0.2in}
\begin{lem} \label{prop:gens-deg3-perm-symm} Let $X$ be a generic symmetric $n\times n$ matrix. For each symmetric $6\times 6$ submatrix $\mathfrak M$ of $X$ we have five minimal generators of degree three in the apolar ideal of the permanent as listed below:

$$\mathfrak M=
\left(%
\begin{array}{cccccc}
 
  a & b & c & d & e & f\\
  b & g & h & i & j & k\\
  c & h & l & m & n & o\\
  d & i & m & p & q & r\\
  e & j & n & q & s & t\\
  f & k & o & r & t & u\\
  
 \end{array}%
\right),$$

$$F_1= EIO-DJO-EHR+CJR+DHT-CIT$$
$$F_2=DKN-DJO-CKQ+BOQ+CJR-BNR$$
$$F_3=FIN-DJO-FHQ+BOQ+CJR-BNR+DHT-CIT$$
$$F_4=EKM-DJO-CKQ+BOQ-EHR+CJR+DHT-BMT$$
$$F_5=FJM-DJO-FHQ+BOQ+DHT-BMT$$

These $5\cdot{n\choose 6}$ polynomials of $Y$ annihilate the permanent of the matrix $X$. 

\end{lem}

\begin{proof}
We use induction on $n$. For $n=6$ it is easy to check that $F_1,\ldots,F_5$ annihilate the permanent of $X$. 
So we assume that for all integer values less than $n$ we have that all the five polynomials coming from the symmetric $6\times 6$ submatrices annihilate the permanent of $X$. We want to show this for $n$. Let $N$ be a $6\times 6$ symmetric submatrix of $Y$, involving the rows and the columns $i_1,\ldots ,i_6$. If a monomial term $\mu$ in the permanent of $X$ is such that $\xi (\mu)\in S^s$ does not contain any of the 15 degree three monomials in the Hafnian of $N$, then $F_1,\ldots,F_5$ annihilate $\mu$. Suppose on the other hand that  $\xi(\mu)$ contains one of the 15 monomials in the Hafnian of $N$.  Since these monomials don't appear in any other $6\times 6$ Hafnian of $Y$, by the first induction step we have shown that $F_1,\ldots,F_5$ annihilate $\mu$. We conclude that $F_1,\ldots, F_5$ annihilate the permanent of $X$. These generators are linearly independent mod $(I_2)_3=(Q')_3$, since they involve different variables, so they are part of a minimal generating set for the apolar ideal of the permanent of $X$.

\end{proof}
The five polynomials $F_1,\ldots ,F_5$ are further discussed in Example \ref{ex:5.6}.

\vspace{0.2in}
\begin{lem} \label{lem:perm-sym-unacceptable-trangularity}
We can write each unacceptable monomial of degree $k$ ($2\leq k\leq n$), as an explicit element of $S^s_{k-2}\langle Q'\rangle$, where $\langle Q'\rangle$ is the space defined in the Definition \ref{def:Wperm-symm}.

\end{lem}

\begin{proof}
We use induction on $k$. For $k=2$ the claim is obviously true.  To show that the claim is true for $k=3$, we need to show that the space $U_3$ of unacceptable monomials in $S^s_3$ is in $S^s_1\langle Q'\rangle$. The unacceptable monomials of degree 3 for the $n\times n$ generic symmetric matrix have one of the following forms:

(a) Unacceptable of the form $x^2y$ where $x$ is a diagonal element. The number of these monomials is $n(\frac{n(n+1)}{2})$.

(b) Unacceptable of the form $xyz$ where $x$ is a diagonal element, $y\neq x$ in the same row or column with $x$ and $z\neq x$. The number of these monomials is $n(n-1)({\frac{n(n+1)}{2}}-1)$.

(c) Unacceptable of the form $xyz$ where $x,y,z$ are non diagonal elements from the same row or column (can be equal to each other). The  number of these monomials is ${n\choose 1}{n-1+3-1\choose 3}$.

Unacceptable monomials of type (a) or (b) are multiples of unacceptable monomials of degree 2, so they are in the space $S^s_1U_2$. So we only need to show that the degree 3 unacceptable monomials of type (c) are in $S^s_1\langle Q'\rangle$. The 3 nondiagonal elements in the same row or column of the matrix $X$ are from a symmetric $4\times 4$ sub-matrix. So without loss of generality we show that a degree 3 monomial of type (c) from the following sub-matrix is in $S^s_1\langle Q'\rangle$. Let $A^s$ be the $4\times 4$ symmetric sub-matrix of a generic symmetric $n\times n$ matrix,$$
A^s=\left(%
\begin{array}{cccc}
 
  a & b & c & d\\
  b & e & f & g\\
  c & f & h & i\\
  d & g & i & j\\

  \end{array}%
\right),$$

and $D^s$ be the matrix

$$
D^s=\left(%
\begin{array}{cccc}
 
  A & B & C & D\\
  B & E & F & G\\
  C & F & H & I\\
  D & G & I & J\\

  \end{array}%
\right)$$

Monomials of type (c) in $S^s_3$ can have one of the following forms. In each case we prove the claim for one monomial in the given form. The proof for any other monomial is similar to what we show below.

(I) All three non-diagonal variables are distinct. The monomials of the form $\eta_1=BCD$ is a degree three unacceptable monomial of type (c). We have $AF-BC\in \langle Q'\rangle$ (since it is a minor with one diagonal element), so we have $(AF-BC)\circ \mathrm{perm}(A^s)=0.$ Hence, $$D(AF-BC)\circ  \mathrm{perm}(A^s)=0.$$ We also know that $DAF\in S^s_1U_2\subset U_3$ so $DAF\circ   \mathrm{perm}(A^s)=0.$ We therefore have $\eta_1=BCD=-D(AF-BC)(\text{mod }S^s_1U_2)$. So we have $\eta_1\in S^s_1 \langle Q'\rangle$. \par

(II) There are two distinct non-diagonal variables. Consider the monomial $\eta_2=B^2C$, also of type (c). We have $AF-BC\in \langle Q'\rangle$ (since it is a minor with one diagonal element), so we have $(AF-BC)\circ  \mathrm{perm}(A^s)=0.$ Hence, $$B(AF-BC)\circ  \mathrm{perm}(A^s)=0.$$ We also know that $BAF\in S^s_1U_2\subset U_3$ so $B^2C\circ   \mathrm{perm}(A^s)=0.$ We therefore have $\eta_2=B^2C=-B(AF-BC)(\text{mod }S^s_1U_2)$. So we have $\eta_2\in S^s_1\langle Q'\rangle$. \par

(III) There is only one non-diagonal variable. Consider the monomial $\eta_3=B^3$, also of type (c). We have $-B^2+2AE\in \langle Q'\rangle$ (since it is a diagonal minor with the coefficient 2), so we have $(-B^2+2AE)\circ \mathrm{perm}(A^s)=0.$ Hence, $$B(-B^2+2AE)\circ \mathrm{perm}(A^s)=0.$$ We also know that $BAE\in S^s_1U_2\subset U_3$ so $B^3\circ \mathrm{perm}(A^s)=0.$ We therefore have $\eta_3=B^3=-B(-B^2+2AE)(\text{mod }S^s_1U_2)$. So we have $\eta_3\in S^s_1\langle Q'\rangle$. \par

So the lemma is true for $k=3$ and we have $U_3\subset S^s_1\langle Q'\rangle$. Let $M$ denote the subspace of $\langle Q'\rangle $ generated by binomials of type (b) and (c) as defined in Proposition \ref{prop:gens-deg2-perm-symm}. We have shown that $U_3\subset S^s_1(U+M)$.

Finally assume that $k\ge 4$ and the lemma is true for all integers less than $k$. We want to show that the claim is true for $k$. Let $\mu=\mu_1\mu_2 \cdots \mu_k$ be an unacceptable monomial of degree $k$. We can write $\mu$ such that $\mu_2 \cdots \mu_k$ is an unacceptable monomial of degree $k-1$ so we have
$$\mu=\mu_1(\mu_2 \cdots \mu_k)\in S^s_1(S^s_{k-3}\langle Q'\rangle)=S^s_{k-2} \langle Q'\rangle,$$
and the lemma is true also for $k$.
\end{proof}
\vspace{.2in}

\begin{defi}\label{def:perm-symmW+} Let $H$ be the ideal generated by the degree three polynomials listed in the Lemma \ref{prop:gens-deg3-perm-symm}. Let $( Q')^+=(Q')+H$ denote the ideal generated by the degree 2 polynomials defined in Proposition \ref{prop:gens-deg2-perm-symm} and the degree 3 polynomials corresponding to the $6\times 6$ symmetric submatrices discussed in the Lemma \ref{prop:gens-deg3-perm-symm}.

\end{defi}

\begin{itemize}
\item The number of $k\times k$ permanents of the $n\times n$ generic symmetric matrix is $$\frac{1}{2}{n\choose k}\cdot {n\choose k}+\frac{1}{2}{n\choose k}, $$ choosing two from a subset of $n\choose k$ elements. These are linearly independent since they involve different variables.

\item Let $\{ P_k\}$ be the set of all $k\times k$ permanents of $X$.
\item As in the determinant case, let $A_k$ be the set of acceptable monomials in $S^s_k$.
\item Let $\iota:R^s \rightarrow S^s,\iota(x_{ij})=y_{ij}$, and let $E_k$ be the subset of $A_k$ defined by $$\{\iota(\mu)|\mu \text{ an initial monomial (in Lex. order) of some element of } \{P_k\}\}.$$
\item Let $E'_k= A_k\backslash E_k$ be the complementary set to $E_k$ in $A_k$.
\item For each $\mu \in A_k$, let $A_{>\mu}$, denote the subset of elements $\nu \in A_k$, such that $\nu>\mu$ in the lexicographic order of $S^s$.
\item Let $[a_1,\ldots ,a_k|b_1,\ldots ,b_k]_p$ be the permanent of the $k\times k$ sub matrix with the rows $\{a_1,\ldots ,a_k\}$ and the columns $\{b_1,\ldots ,b_k\}$. Recall that for a monomial $$m=y_{a_1b_1}\cdots y_{a_kb_k}=(a_1,\ldots ,a_k|b_1,\ldots ,b_k),$$
we call a pair $(b_i,b_j)$, with $i<j$, a reversal pair if $b_i\ge b_j$. The initial term of the $k\times k$ permanent $[a_1,\ldots ,a_k|b_1,\ldots ,b_k]_p$ is the term $y_{i_1,j_1}y_{i_2,j_2}\cdots y_{i_k,j_k}$ such that $i_1\leq i_2 \leq \cdots \leq i_k$ and $j_1\leq j_2 \leq \cdots \leq j_k$ where $\{i_1,\ldots ,i_k\}=\{a_1,\ldots ,a_k\}$ and $\{b_1,\ldots,b_k\}=\{j_1,\ldots ,j_k\}$.

\end{itemize}

\vspace{.2in}
\begin{prop} \label{prop:main-triangularity-perm-symm}Each acceptable monomial in $E'_k$ ($3\leq k\leq n$), is the initial monomial (in the reverse Lex. order) of an element of $( Q')^+_k$. 
\end{prop}

\begin{proof}
We use induction on $k$, and start with $k=3$.  Let $\mu$ be a degree 3 acceptable monomial which is not the initial term of any $3\times 3$ permanent in the lexicographic order. The acceptable monomials, $x_{i_1i_2}x_{i_3i_4}x_{i_5i_6}$, of degree 3 for the $n\times n$ generic symmetric matrices can be listed as follows:

(a) All 6 indices are distinct.

(b) There is one repeated index.

(c)  There are 2 repeated indices.

(d) There are 3 repeated indices.

We discuss each of the above types separately in each case for one monomial of the given form. The proof for any other monomial of the given type is similar to what we show.

(a) all 6 indices are distinct $x_{i_1i_2}x_{i_3i_4}x_{i_5i_6}$, $(i_1,i_3,i_5|i_2,i_4,i_6)$. Without loss of generality we can assume these indices are 1,2,3,4,5,6. In order to have a non-initial monomial of this kind, it is enough to have at least one reversal pair. For example in $(1,2,3|6,4,5)$, $6\ge 4$ so $(6,4)$ is a reversal pair. 

In the doset minor $(i_1,i_3,i_5|i_2,i_4,i_6)$, without loss of generality we may arrange that
$$i_1<i_3<i_5.  $$
Now assume   $i_1=1,i_3=2$, $i_5=3$ $i_2=4,i_4=5$ and $i_6=6$ then $x_{14}x_{25}x_{36}$ is the initial term in the corresponding $3\times 3$ permanent using the lexicographic order. So in order to have a non-initial monomial we need to assign to  $i_2,i_4$ and $i_6$ the numbers 4,5 and 6 but not in order. So we have at least one reversal pair $(i_j,i_k),(j<k)$ where $j,k\in\{2,4,6\}$ such that $i_j\ge i_k$.

Next we look at the corresponding $6\times 6$ symmetric sub-matrix. we have 
$$
X=\left(%
\begin{array}{cccccc}
 
  a & b & c & d & e & f \\
  b & g & h & i & j & k  \\
  c & h & l & m & n & o\\
  d & i & m & p & q & r\\
  e & j & n & q & s & t\\
  f & k & o & r & t & u\\
  
   \end{array}%
\right),$$

$$
Y=\left(%
\begin{array}{cccccc}
 
   A & B & C & D & E & F \\
  B & G & H & I & J & K  \\
  C & H & L & M & N & O\\
  D & I & M & P & Q & R\\
  E & J & N & Q & S & T\\
  F & K & O & R & T & U\\

   \end{array}%
\right),$$

Consider a degree three non-initial monomial, the terms coming from 6 distinct rows and columns. A general example of this kind is $(1,2,5|6,4,3)$. But we have:

$$
FIN-DJO-FHQ+BOQ+CJR-BNR+DHT-CIT\in H,
$$

so given $\mu=FIN$,  we have

$$f_{\mu}=FIN-DJO-FHQ+BOQ+CJR-BNR+DHT-CIT \in ( Q')^+_3,$$ 

where 

$$-DJO-FHQ+BOQ+CJR-BNR+DHT-CIT\in A_{>\mu}.$$

(b) There is one repeated index. Without loss of generality we can assume the indices are 1,2,3,4,5, with one of them repeated. In order to have a non-initial example of this kind, it is enough to have 1 reversal pair. For example in $(1,2,3|4,1,5)$, $4\ge 1$.  We can form a $5\times 5$ symmetric matrix with these rows and columns,

$$
X=\left(%
\begin{array}{ccccc}
 
  a & b & c & d & e \\
  b & f & g & h & i \\
  c & g & j & k & l \\
  d & h & k & m & n\\
  e & i & l & n & o \\

   \end{array}%
\right),$$

$$
Y=\left(%
\begin{array}{ccccc}
 
  A & B & C & D & E \\
  B & F & G & H & I \\
  C & G & J & K & L \\
  D & H & K & M & N\\
  E & I & L & N & O \\

   \end{array}%
\right),$$

Now consider the monomial $\mu=y_{14}y_{21}y_{35}=BDL$ as an acceptable monomial of type (b). We have the minor $AH-BD\in W$. Given $\mu=BDL$, $f_{\mu}=-L(AH-BD)\in S^s_1(Q')\subset (Q')^+_3, \text{ where } AHL \in A_{>\mu}.$

(c) There are 2 repeated indices, Without loss of generality we can assume the indices are 1,2,3,4, with two of them repeated. In order to have a non-initial example of this kind, it is enough to have 1 reversal pair. For example in $(1,2,3|2,1,4)$, $2\ge 1$. We can form a $4\times 4$ symmetric matrix with these rows and columns,

$$
X=\left(%
\begin{array}{cccc}
 
  a & b & c & d  \\
  b & e & f & g \\
  c & f & h & i \\
  d & g & i & j \\

   \end{array}%
\right),$$

$$
Y=\left(%
\begin{array}{cccc}
 
  A & B & C & D  \\
  B & E & F & G \\
  C & F & H & I \\
  D & G & I & J \\

   \end{array}%
\right),$$

and consider the monomial $\mu=y_{12}y_{21}y_{34}=B^2I$ as an acceptable monomial of type (c). Given $\mu=B^2I$, $f_{\mu}=I(B^2-2AE)\in S^s_1(Q')\subset (Q')^+_3, \text{ where } AEI \in A_{>\mu}.$

(d) There are 3 repeated indices, Without loss of generality we can assume the indices are 1,2,3, each of them repeated. In order to have a non-initial example of this kind, it is enough to have 1 reversal pair. For example in $(1,2,3|3,2,1)$, in the second column of the tableau $3\ge 1$. So we can form a $3\times 3$ symmetric matrix with these rows and columns,
$$
X=\left(%
\begin{array}{ccc}
   a & b & c  \\
  b & d & e  \\
  c & e & f \\
    \end{array}%
\right),$$

$$
Y=\left(%
\begin{array}{ccc}
 
  A & B & C  \\
  B & D & E  \\
  C & E & F \\

   \end{array}%
\right).$$

The monomial $\mu=y_{13}y_{22}y_{31}=C^2D$ is an acceptable monomial of type (c). Given $\mu=C^2D$, $f_{\mu}=-C(BE-CD)\in S^s_1(Q')\subset (Q')^+_3, \text{ where } BEC \in A_{>\mu}.$

So the claim of the Proposition is true for $k=3$.

Now assume $k>3$ and that the Proposition \ref{prop:main-triangularity-perm-symm} is true for all integers less than $k$. We want to show that the Proposition \ref{prop:main-triangularity-perm-symm} is also true for $k$. Let $\mu=\mu_1\mu_2 \cdots \mu_k$ be a degree $k$ acceptable non-initial monomial, we can write $\mu$ such that $\mu_2 \cdots\mu_k$ is a degree $k-1$ acceptable non-initial monomial (it is enough that the monomial $\mu_2 \cdots \mu_k$ includes one reversal). Then by the induction assumption  $\mu_2 \cdots \mu_k$ is the initial monomial (in rev. lex.) of an element of $(Q')^+_{k-1}$. So we have $\mu=\mu_1(\mu_2 \cdots \mu_k)$ is the initial monomial (in rev. lex.) of an element of $(Q')^+_k$. This completes the proof.
\end{proof}
\vspace{.2in}

\begin{cor} \label{cor:main-k-perm-symm}
For $2 \leq k \leq n$ we have
\begin{equation*}
(Q')^+_k=\mathrm{Ann}(\mathrm{perm}(X))\cap S^{s}_{k}
\end{equation*}

\end{cor}

\begin{proof}
We have

(A) $(Q')^{+} \circ \mathrm{perm} (X)=0  \Longleftrightarrow (Q')^+ \circ S^s_{n-2}\circ (\mathrm{perm}(X))=0 \Longleftrightarrow  (Q')^+ \circ P_2(X)=0$.

(B) \text{ By Lemma \ref{lem:perm-symm-W-degree2} we have} ($\mathrm {Ann}(\mathrm{perm}(X))) \cap S^s_2= (Q')^+_2\Rightarrow S^s_{k-2} ((Q')^+)\circ (S^s_{n-k} \circ \mathrm{perm} (X))=0$.
\newline$ \Rightarrow S^s_{k-2}((Q')^+) \circ P_k(X)=0$.
\newline$ \Rightarrow ((Q')^+)_k \circ P_k(X)=0$. (By Remark \ref{remark:introIK})

Therefore

\begin{equation*}
(Q')^+_k \subset \mathrm {Ann}(P_k(X))  \cap S^s_k.
 \end{equation*}

By Remark \ref{remark:introIK} and Lemma \ref{lem:Scircleperm-symm} we have
\begin{equation*}
(\mathrm{Ann}(\mathrm{perm}(X)))_k=(\mathrm{Ann}(S^s_{n-k}\circ(\mathrm{perm}(X)))_k=(\mathrm{Ann}(P_{k}(X)))_k
\end{equation*}

So we have

\begin{equation*}
\dim (Q')^+_k \leq \dim(\mathrm{Ann}(\mathrm{perm}(X))\cap S^s_k)=\dim S^s_k-\dim P_k(X).
\end{equation*}

On the other hand by the Definition \ref{def:perm-symmW+} the sets $U_k$ and $E'_k$ are linearly independent and form a basis for the corresponding subspaces. So by Lemma \ref{lem:perm-sym-unacceptable-trangularity} and Proposition \ref{prop:main-triangularity-perm-symm}, we have

$$
\dim \langle Q'\rangle ^+_k\ge\dim S^s_k-\dim P_k(X)=\dim <E'_k>+\dim <U_k> .
$$

So we have
$$
\dim \langle Q' \rangle ^+_k=\dim S^s_k-\dim P_k(X)=\dim( \mathrm{Ann}(\mathrm{perm}(X))\cap S^s_k).
$$

\end{proof}

\vspace{.2in}

\begin{thm}\label{thm:perm-sym-main-theorem}
Let $X$ be a generic symmetric $n\times n$ matrix. Then the apolar ideal $\mathrm{Ann}(\mathrm{perm}(X))$ is the ideal $(Q')^+$ of Definition \ref{def:perm-symmW+}, generated in degrees two and three.
\end{thm}

\begin{proof}
This follows directly from Proposition \ref{prop:main-triangularity-perm-symm} and Corollary \ref{cor:main-k-perm-symm}.
\end{proof}
\vspace{.2in}


\section{Application to the ranks of the determinant and permanent of the generic symmetric matrix}
\label{section:ranks-symmetric}

In this section we apply our results from sections \ref{section: Dimension} and \ref{Section:symmetric-Generators of the apolar ideal } to find some lower bounds for the the cactus rank and the rank of the determinant and permanent of the generic symmetric matrix. 

\begin{Notation}
Let $F\in R^s=\sf k$$[x_{ij}]$ be a homogeneous form of degree $d$. A presentation 
\begin{equation}\label{eq:Waring decomposition}
F=l_1^d+\cdots +l_s^d \text{ with }l_i\in R^s_1.
\end{equation}

is called a \emph{Waring decomposition} of length $s$ of the polynomial $F$. The minimal number $s$ that satisfies the Equation \ref{eq:Waring decomposition} is called the \emph{rank} of $F$.

Let $F\in R^s=\sf k$$[x_{ij}]$ be a homogeneous form. The apolarity action of $S^s=\sf k$$[y_{ij}]$ on $R^s$, defines $S^s$ as a natural coordinate ring on the projective space $\textbf{P}(R^s_1)$ of 1-dimensional subspaces of $R^s_1$. A finite subscheme $\Gamma\subset \textbf{P}(R^s_1)$ is apolar to $F$ if the homogeneous ideal $I_\Gamma \subset S^s$ is contained in $\mathrm {Ann}(F)$ (\cite{IK},\cite{RS}).
\vspace{0.2in}

\end{Notation}
\begin{remark}\label{rmrk:rank-definition-ranestad}((\cite{IK} Def. 5.66,\cite{RS})) Let $\Gamma=\{[l_1],\ldots,[l_s]\}$ be a collection of $s$ distinct points in $\textbf{P}(R_1)$. Then 
$$
F=c_1l_1^d+\cdots+c_sl_s^d \text{ with }c_i\in\sf{k}.
$$
if and only if 
$$I_\Gamma \subset \mathrm {Ann}(F)\subset S$$

\end{remark}

\vspace{0.2in}

\begin{defi}\label{def:ranks} 

We have the following ranks (\cite{IK} Def. 5.66 , \cite{BR} and  \cite{RS}). Here $\Gamma$ is a punctual scheme (possibly not smooth), and the degree of $\Gamma$ is the number of points (counting multiplicities) in $\Gamma$.

a. the rank $r(F)$:
\begin{equation*}
r(F)=\min\{\deg \Gamma| \Gamma\subset \textbf{P}(R^s_1) \text{ smooth}, \dim \Gamma=0,I_\Gamma \subset \mathrm {Ann}(F)\}.
\end{equation*}

Note that when $\Gamma$ is smooth, it is the set of points in the Remark \ref{rmrk:rank-definition-ranestad} (\cite{IK}, page 135).

b. the smoothable rank  $sr(F)$:
\begin{equation*}
sr(F)=\min\{\deg \Gamma| \Gamma\subset \textbf{P}(R^s_1) \text{ smoothable}, \dim \Gamma=0,I_\Gamma \subset \mathrm {Ann}(F)\}.
\end{equation*}

Note that for the smoothable rank one considers the smoothable schemes, that are the schemes which are the limits of smooth schemes of $s$ simple points (\cite{IK}, Definition 5.66). 

c. the cactus rank (scheme length in \cite{IK}, Definition 5.1 page 135) $cr(F)$:
\begin{equation*}
cr(F)=\min\{\deg \Gamma| \Gamma\subset \textbf{P}(R^s_1), \dim \Gamma=0,I_\Gamma \subset \mathrm {Ann}(F)\}.
\end{equation*}

d. the differential rank (Sylvester's catalecticant or apolarity bound) is the maximal dimension of a homogeneous component of $S^s/\mathrm {Ann}(F)$:
\begin{equation*}
l_{diff}(F)= \max_{i\in \mathbb{N}_0}  \{ (H(S^s/\mathrm {Ann}(F)))_i\}.
\end{equation*}
\end{defi}

Below we give a lower bound for the cactus rank of the determinant and permanent of the generic symmetric matrix. We do not have information on  the smoothable rank of the generic symmetric determinant or permanent. It is still open to find good bounds for the smoothable rank. The work of A. Bernardi and K. Ranestad \cite{BR} in the case of generic forms of a given degree and number of variables show that the cactus rank and smoothable rank can be very different.

\vspace{0.2in}
\begin{prop}\label{prop:ranks-inequality}(\cite{IK}, Proposition 6.7C)
The above ranks satisfy
\begin{equation*}
l_{diff}(F) \leq cr(F) \leq sr(F) \leq r(F).
\end{equation*}
\end{prop}

\begin{prop}\label{prop-mainRS} \textbf{(Ranestad-Schreyer)} (\cite{RS}, Proposition 1) If the ideal of $\mathrm {Ann}(F)$ is generated in degree d and $\Gamma\subset \textbf{P}(R^s_1)$ is a finite (punctual) apolar subscheme to $F$, then
\begin{equation*}
\deg \Gamma \geq \frac{1}{d} \deg (\mathrm {Ann}(F)),
\end{equation*}
where $\deg (\mathrm {Ann}(F)) =\dim (S^s/\mathrm {Ann} (F))$ is the length of the 0-dimensional scheme defined by $\mathrm {Ann}(F)$. 
\end{prop}
\vspace{.2in}

\begin{remark}
The Ranestad-Schreyer Proposition is true for arbitrary characteristic: the argument depends on B\'{e}zout's theorem, which is true for $\sf k$ algebraically closed (see \cite{Go}, page 113); and none of the degrees involved in the proof changes as one extends from an infinite field $\sf k$ to its algebraic closure.
\end{remark}

Using the Ranestad-Schreyer Proposition \ref{prop-mainRS} and our results in sections \ref{section: Dimension} and \ref{Section:symmetric-Generators of the apolar ideal }, we have
\vspace{0.2in}

\begin{thm}\label{thm:RS-rank-symm-det}
For the determinant  of a generic symmetric $n\times n$ matrix $X$, we have
\begin{equation*}
{1\over{2(n+2)}} {2n+2\choose n+1} \leq cr(\det(X)) \leq sr(\det(X)) \leq r(\det(X)).
\end{equation*}
\end{thm}
\begin{proof}
This follows directly from Propositions  \ref {prop:ranks-inequality} and \ref{prop-mainRS}, Theorem \ref{main-theorem-det-symmetric} and Corollary \ref{cor:degree-catalan}.

\end{proof}

\begin{thm}\label{thm:RS-rank-symm-perm}
For the permanent  of a generic symmetric $n\times n$ matrix $X$, we have
\begin{equation*}
\frac{{2n\choose n}+2^n}{6} \leq cr(\mathrm{perm}(X)) \leq sr(\mathrm{perm}(X)) \leq r(\mathrm{Perm}(X)).
\end{equation*}
\end{thm}

\begin{proof}
This follows directly from Proposition \ref{prop:ranks-inequality} and \ref{prop-mainRS}, Theorem \ref{thm:perm-sym-main-theorem} and Table \ref{table-symm-hilb-perm}.

\end{proof}

\begin{Notation} \cite{LT}
Let $\Phi \in S^d\mathbb C^n$ be a degree $d$ polynomial with coefficients in the complexes $\mathbb C$, in variables $(x_1,\ldots ,x_n)$, we can polarize $\Phi$ and consider it as a multilinear form $\tilde{\Phi}$ where $\Phi(x)=\tilde{\Phi}(x,\ldots,x)$ and consider the linear map $\Phi_{s,d-s}:S^s\mathbb C^{n*}\rightarrow S^{d-s}\mathbb C^{n}$, where $\Phi_{s,d-s}(x_1,\ldots,x_s)(y_1,\ldots,y_{d-s})=\tilde{\Phi}(x_1,\ldots ,x_s,y_1,\ldots ,y_{d-s})$. Define
\begin{equation*}
Zeros(\Phi)=\{[x]\in \mathbb P \mathbb C^{n*}| \Phi(x)=0\} \subset  \mathbb P \mathbb C^{n*}.
\end{equation*}
Let $x_1,\ldots,x_n$ be linear coordinates on  $\mathbb C^{n*}$ and define

$$
\Sigma_s(\Phi):=\{[x] \in Zeros(\Phi)| \frac{\partial^I \Phi}{\partial x^I}(x)=0,\forall  I,\text{ such that } |I|\leq s\}.
$$

\end{Notation}
\vspace{.2in}

In this notation $\Phi_{s,d-s}$ is the map from $S_s\to R_{n-s}$ taking  $h$ to $h\circ \Phi$, hence its rank is  $H(\mathfrak A_A)_s$.

In the following theorem we use the convention that $\dim \emptyset=-1$.
\vspace{.2in}

\begin{thm} \label{thm:LT-rank}\textbf{(Landsberg-Teitler)}(\cite{LT}, Theorem 1.3) 
Let $\Phi \in S^d\mathbb C^{n}$, Let $1\leq s \leq d$. Then 

\begin{equation*}
rank(\Phi)\geq rank \Phi_{s,d-s}+ \dim \Sigma_s(\Phi)+1.
\end{equation*}
\end{thm}
\vspace{.2in}
\begin{rmk} (\textbf{Z. Teitler})
If we define $\Sigma_s(\Phi)$ to be a subset of affine rather than projective space, then the above theorem  does not need $+1$ at the end, and does not need the statement that the dimension of the empty set is $-1$.

\end{rmk}
\vspace{.2in}

Using the Landsberg-Teitler formula we have:
\vspace{.2in}

\begin{prop}\label{prop-det-symmetric-result}
Let $X$ be a generic symmetric $n \times n$ matrix. For each $t$, $1\leq t\leq n$, we have

\begin{equation}\label{prop-det-symmetric-result-equation}
r(\det(X))\geq \frac{{n+1 \choose t}{n+1\choose t+1}}{n+1}+\frac{(n-t-1)(n-t)}{2}+(t+1)(n-t-1)+1.
\end{equation}

Asymptotically for large $n$ the maximum of the right hand side of Equation (\ref{prop-det-symmetric-result-equation}) occurs at $t=\lfloor n/2\rfloor$, and asympotically $r(\det X)\ge \frac{2^n}{n\sqrt{n\pi}}$.

\end{prop}

\begin{proof}
By Lemma 2.5 the dimension of the space of $t\times t$ minors of $X$ is the Narayana number $\frac{{n+1 \choose t}{n+1\choose t+1}}{n+1}$. The determinant of $X$ vanishes to order $t+1$ if and only if every minor of $X$ of size $n-t$ vanishes. Thus $\Sigma_t(\det_n)$ is the locus of matrices of rank at most $n-t-1$ so the $\dim \Sigma_t(\det_n)$ is $\frac{(n-t-1)(n-t)}{2}+(t+1)(n-t-1)$ (\cite{BH}, page 304). By the unimodality of the binomial coefficients the first term of the right hand side of Equation (\ref{prop-det-symmetric-result-equation}) is maximum at  $t=\lfloor {n/2}\rfloor$. The other terms in the right hand side of Equation (\ref{prop-det-symmetric-result-equation}) consist of a degree two polynomial in $t$ and is decreasing for $1\leq t$. So the asymptotic maximum of the right hand side is for $t=\lfloor {n/2}\rfloor$. A calculation using Stirling's formula gives the asymptotic lower bound.
\end{proof}
\vspace{0.2in}

\begin{ex}

Let $X$ be a $4\times 4$ generic symmetric matrix. The Hilbert sequence corresponding to the ideal $\mathrm {Ann}(\det(X))$ will be $H=(1,10,20,10,1)$. Using the Ranestad-Schreyer Proposition \ref{prop-mainRS} we have:
\begin{equation*}
\deg \Gamma \geq \frac{1}{d} \deg (\mathrm {Ann}(\det(X)))=\frac{1}{2} (42)=21.
\end{equation*}
Now using the Proposition \ref{prop-det-symmetric-result} we have
\begin{equation*}
r({\det}_4)\geq \frac{{4+1 \choose 2}{4+1\choose 2+1}}{4+1}+\frac{(4-2-1)(4-2)}{2}+(2+1)(4-2-1)+1=25,
\end{equation*}
which is a better lower bound. However, as $n$ increases the Ranestad-Schreyer Proposition \ref{prop-mainRS} gives better lower bounds for the cactus rank of the determinant than Proposition \ref{prop-det-symmetric-result}.

\end{ex}

Table \ref{Table:rank-symm-det} gives the lower bounds for the cactus rank (using RS bound) and rank (using LT) bound, and also the Sylvester rank of the determinant of an $n\times n$ generic symmetric matrix $X$, for $2\le n\le 6$ and also for $n\gg 0$, using the Stirling formula with Theorem \ref{thm:RS-rank-symm-det} and Proposition \ref{prop-det-symmetric-result}. Asympotically the RS lower bound for cactus rank is approximately $2^{n+1}$ times that for the rank from LT, so is significantly stronger in the light of Proposition 4.3.

\begin{table}[h]
\begin{center}
\caption{The determinant of the generic symmetric matrix}\label{Table:rank-symm-det}

\begin{tabular}{l*{7}{c}r}
\hline
$n$              & 2 & 3 & 4 & 5 & 6 & $n\gg 0$ \\
\hline
lower bound for $cr(\det(X))$ using RS & 2.5 & 7 & 21 & 66 & 209.5  &$ \frac{2^{2n+1}}{(n+1)\sqrt{(n+1)\pi}}$\\
\hline
lower bound for $r(\det(X))$ using LT           & 4 & 7 & 25 &  56 & 187 &$ 2^n/n\sqrt{n\pi}$\\
\hline
$l_{diff}(\det_n)$          & 3 & 6 & 20 &  50 & 175 & ${n\choose {\lfloor n/2\rfloor}}^2/ {\lfloor n/2\rfloor}$\\
\hline
\end{tabular}
\end{center}
\end{table}

H. Derksen and Z. Teitler in subsequent work have shown lower bounds for the cactus rank of $\det (X)$ that are asymptotically $3/2$ times those from Theorem \ref{thm:RS-rank-symm-det}. \cite[Example 4.3]{DT}\par

The following table gives the lower bounds for the cactus rank of the permanent of an $n\times n$ generic symmetric matrix $X$, for $2\leq n\leq 6$, and also for $n\gg 0$ using the Stirling formula and Theorem \ref{thm:RS-rank-symm-perm}, and we give also the Sylvester lower bound.

\begin{table}[h]
\begin{center}
\caption{The permanent of the generic symmetric matrix} \label{Table:rank-symm-perm}

\begin{tabular}{l*{7}{c}r}
\hline
$n$              & 2 & 3 & 4 & 5 & 6 &  $n\gg 0$ \\
\hline
lower bound for $cr(\mathrm{perm}(X))$ using RS & 1.6 &  4.6 & 14.3 & 47.3 & 164.6   &$ \frac{2^n(2^n+\sqrt{\pi n})}{6\sqrt{\pi n}}$\\

\hline
$l_{diff}(\mathrm{perm}(X))$          & 3 & 6 & 21 &  55 & 210 &$\frac{\sqrt{2}}{\sqrt{\pi n}} 2^{n-1}$\\
\hline

\end{tabular}
\end{center}
\end{table}

Note that for $n\leq 8$ and $n=10$ the $l_{diff}(\mathrm{perm}(X))$ is a larger lower bound. For $n=9$ and $n\ge 11$ our result using RS is larger than the $l_{diff}(\mathrm{perm}(X))$.
\vspace{.2in}


\section{Hafnian Invariants}
\label{inv}

In this section, we explore some facts about the degree three generators of the ideal $\mathrm{Ann}(\mathrm{perm}(X))$ for a generic symmetric matrix. The goal is to understand the degree three generators of the $\mathrm{Ann}(\mathrm{perm}(X))$, for a generic symmetric matrix (see Lemma \ref{prop:gens-deg3-perm-symm}). 
\vspace{0.2in}
\begin{Notation}
Throughout this section, $\mathfrak{S}_n$ is the symmetric group of order $n$, $X=(x_{ij})$ is a generic symmetric matrix in the polynomial ring $R^s$ and $Y$ is a generic symmetric matrix in the corresponding differential operator ring $S^s$. $\sigma\in\mathfrak{S}_n$ acts on $R^s=k[x_{ij}]$ as follows:
\begin{equation}\label{eq:action of Sn on vars}
\sigma(x_{ij})=x_{\sigma^{-1}(i)\sigma^{-1}(j)}.
\end{equation}
Let $\mathrm{MonHaf}_{2k}(X)$ denote the space of the monomials of the hafnian (Section 3 Equation (\ref{eq:hafnian-definition})) of a $2k\times 2k$ generic symmetric matrix $X$. 
\end{Notation}
\vspace{0.2in}
\begin{remark}\label{remark:new definition of hafnian}\cite{Vi} Let $X$ be a $2k\times 2k$ symmetric matrix. We can write the hafnian of $X$ as 
\begin{equation}\label{eq:new-def-Haf}
\mathrm{Hf}(X)=\frac{1}{2^k\cdot k!}\sum_{\sigma\in \mathfrak{S}_{2k}} x_{\sigma(1)\sigma(2)}x_{\sigma(3)\sigma(4)} \cdots x_{\sigma(2k-1)\sigma(2k)}.
\end{equation}
\end{remark}
By Equation \ref{eq:new-def-Haf} it is easy to see that $\mathrm{Hf}(X)$ is invariant under $\mathfrak{S}_{2k}$, and we have
\begin{equation}\label{eq:dimension of Haf space}
\dim_{\sf k} \mathrm{MonHaf}_{2k}(X)=\frac{(2k)!}{2^k\cdot k!}
\end{equation}

\begin{defi}\label{def:dual-maps} Let $n=2k$, and $X$ be a generic symmetric $n\times n$ matrix with variables in $R^s$. Recall that $P_k(X)$ is the space of permanents of $k\times k$ submatrices of $X$, and $M_k(X)$ the space of $k\times k$ minors of $X$.  We define the maps $\Phi$ and $\Psi$ as follows

$$ \Phi:\mathrm{MonHaf}_{2k}(Y)\longrightarrow S_k\circ \mathrm{perm}(X)=P_k(X),$$
$$ \Phi(h)=h\circ \mathrm{perm}(X).$$

 Let $\Omega$ denote the kernel of the map $\Phi$, and
 
 $$\Psi:\mathrm{MonHaf}_{2k}(Y)\longrightarrow S_k\circ \det (X)=M_k(X),$$
 $$\Psi(h)=h\circ \det(X).$$

 \end{defi}
 
 \begin{ex}
 Let $n=4$,
  $$X=
\left(%
\begin{array}{cccc}
 
  a & b & c & d \\
  b & e & f & g \\
  c & f & h& i \\
  d & g & i & j \\

 \end{array}%
\right),
Y=
\left(%
\begin{array}{cccc}
 
  A & B & C & D \\
  B & E & F & G \\
  C & F & H& I \\
  D & G & I & J \\

 \end{array}%
\right).
$$
In this case $\mathrm{Hf}(Y)=BI+CG+DF$ so we have $\dim_{\sf k} \mathrm{MonHaf}_{4}(Y)=3$. Consider the map $$\Phi:\mathrm{MonHaf}_{4}(Y)\longrightarrow S_2\circ \mathrm{perm}(X)=P_2(X)$$
 $$\Phi(h)=h\circ \mathrm{perm}(X).$$  
 We have
  $$\mathrm{Im}(\Phi)=\langle(2df + 2cg + 4bi, 2df + 4cg + 2bi, 4df + 2cg + 2bi)\rangle$$$$=\langle(bi, cg, df)\rangle$$
  
Hence $\dim_{\sf k} \mathrm{Im}(\Phi)=3$. So $\mathrm{Ker} \Phi=0$.

 Now let $$\Psi:\mathrm{MonHaf}_{4}\longrightarrow S_2\circ \det (X)=M_2(X)$$
 $$\Psi(h)=h\circ \det(X).$$ 
 The kernel of the map $\Psi$ is $\mathrm{MonHaf}_4\cap\mathrm{Ann}(\det(X))$.

We have $$\mathrm{Im}\Psi=\langle(- 2df - 2cg + 4bi, - 2df + 4cg - 2bi, 4df - 2cg - 2bi)\rangle$$
$$=\langle(cg - bi, df - bi)\rangle$$

is a two dimensional space. The kernel of $\Psi$ is $\langle BI+CG+DF\rangle$, which is a one dimensional space. Note also that
$$(\mathrm{Ann}(\mathrm{Hf}(X)))_2=(CG - BI, DF - BI)$$

Table \ref{table:s4hafchar} shows the character table of $\mathfrak{S}_4$ acting on the space $\mathrm{MonHaf}_4$.
\begin{table}[h]
\begin{center}
\caption{$\mathfrak{S}_4$/hafnian}\label{table:s4hafchar}

\begin{tabular}{l*{5}{c}r}
number of elements in the conjugacy class         & 1 & 6 & 8 & 6  & 3 \\

conjugacy class & $(1^4)$ & $(1^22)$  & (13) & (4) & $(2^2)$ \\
\hline
$\chi_{\mathrm{MonHaf}_4}(g)$ & 3 & 1 &  0 & 1 & 3 \\
\hline

\end{tabular}
\end{center}
\end{table}

Here we have $3^2(1)+6+3(3^2)+6=48$. So $|\chi(\mathrm{MonHaf}_4)|=2$. This representation is the sum of two irreducible representations of $\mathfrak{S}_4$ (see \cite{FH}, page 17). Indeed the character table of $\mathfrak{S}_4$, Table \ref{character table of s4}, shows that $\mathrm{MonHaf}_4=U\oplus W$ corresponding to the partitions $[4]$ and $[2,2]$. Hence the image of $\Psi$ is the irreducible representation $W$ and the Kernel of $\Psi$ corresponds to the trivial representation $U$. We also see that the image of $\Phi$ corresponds to the representation of $\mathrm{MonHaf}_4=U\oplus W$.

\begin{table}[h]
\begin{center}
\caption{character table of $\mathfrak{S}_4$}\label{character table of s4}

\begin{tabular}{l*{5}{c}r}
number of elements in the conjugacy class               & 1 & 6 & 8 & 6 & 3  \\

conjugacy class & $(1^4)$ & $(1^22)$ & $(13)$ & (4) & $(2^2)$  \\
\hline
trivial $U$ & 1 & 1 & 1 &  1 & 1 \\
\hline
alternating $U'$ & 1 & -1 & 1 &  -1 & 1 \\
\hline
standard $V$ & 3 & 1 & 0 &  -1 & -1 \\
\hline
$V'=V\otimes U'$ & 3 & -1 & 0 &  1 & -1 \\
\hline
 $W$ & 2 & 0 & -1 &  0 & 2 \\
\hline

\end{tabular}
\end{center}
\end{table}
 \end{ex}
 
 \vspace{0.2in}
 
\begin{lem} \label{lem:equivariant-n=6} For $n=6$, the map $\Phi$ is an $\mathfrak{S}_6$ equivariant map.
\end{lem}
\begin{proof}
Let $\mathfrak{h}\in \mathrm{Haf}_6(Y)$. Let
$$
X=\left (%
\begin{array}{cccccc}
 
  a & b & c & d & e & f\\
  b & g & h & i & j & k\\
  c & h & l & m & n & o\\
  d & i & m & p & q & r\\
  e & j & n & q & s & t\\
  f & k & o & r & t & u\\
  
 \end{array}%
\right),$$
and let $Y$ be the the corresponding matrix in the ring of differential operators in the variables $A,...,U$. Without loss of generality we can take $\mathfrak{h}=BMT$. We want to show that for any element $\sigma \in\mathfrak{S}_6$ we have:

$$\Phi(\sigma\cdot \mathfrak{h})=\sigma\cdot \Phi(\mathfrak{h})=\sigma\cdot (\mathfrak{h}\circ \mathrm{perm}(X)).$$

Let $\sigma$ be a transposition in $\mathfrak{S}_6$. Without loss of generality we can assume $\sigma=(1 4)$. We have 
$$(14)\cdot BMT=CIT$$

Hence we need to show:
$$\Phi(CIT)=(14)\cdot (BMT\circ \mathrm{perm}(X)).$$

We have:

$$
CIT\circ\mathrm{perm}(X)=2fjm + 2ekm + 4fin + 2dkn + 4eio + 2djo + 2fhq +$$$$ 4ckq + 2boq + 2ehr + 4cjr+ 2bnr + 4dht + 8cit + 4bmt.$$

We also have

$$
BMT\circ\mathrm{perm}(X)=4fjm + 4ekm + 2fin + 2dkn + 2eio + 2djo + 2fhq +2ckq $$

Hence 

$$(14)\cdot(BMT\circ\mathrm{perm}(X))=CIT\circ\mathrm{perm}(X).$$
\end{proof}

\begin{ex}\label{ex:5.6}
Let $n=6$, and $X$ be the generic symmetric matrix in the proof of Lemma \ref{lem:equivariant-n=6}. In this case $\dim_{\sf k}\mathrm{MonHaf}_{6}(Y)=15$. The kernel of the map $$\Phi:\mathrm{MonHaf}_{6}(Y)\longrightarrow S_3\circ \mathrm{perm}(X)=P_3(X)$$
 $$\Phi(\mathfrak{h})=\mathfrak{h}\circ \mathrm{perm}(X).$$  is a five dimensional space  $\Omega=<F_1,\ldots ,F_5>$, where  $F_1,\ldots, F_5$ are the polynomials introduced in the Lemma \ref{prop:gens-deg3-perm-symm}. Hence $\dim_{\sf k} \Omega=5$, and $\dim_{\sf k} \mathrm{Im}\Phi=10$.
 
Now let $$\Psi:\mathrm{MonHaf}_{6}(Y)\longrightarrow S_3\circ \det (X)=M_3(X)$$
 $$\Psi(\mathfrak{h})=\mathfrak{h}\circ \det(X).$$ The kernel of $\Psi$ is $\mathrm{MonHaf}_6(Y)\cap\mathrm{Ann}(\det(X))$.Using Macaulay 2 for calculations we have $$\mathrm{Im}\Psi=\langle F_1(X),\ldots ,F_5(X)\rangle,$$ is a five dimensional space. Hence the kernel of $\Psi$ is a 10-dimensional space.

The character table of $\mathfrak{S}_6$ acting on the space $\mathrm{MonHaf}_6$ is shown in Table \ref{table:s6hafchar}.
\begin{table}[h]
\begin{center}
\caption{$\mathfrak{S}_6$/hafnian}\label{table:s6hafchar}

\begin{tabular}{l*{11}{c}r}
\hline
 number of elements\\ in the class        & 1 & 15 & 40 & 45 & 90 & 120 & 144 & 15 & 90 & 40 & 120  \\
 \hline
conjugacy class & $(1^6)$ & $(1^42)$ & $(1^33)$ & $(1^22^2)$ & $(1^24)$ & (123) & (15) & $(2^3)$ & (24) & $(3^2)$ & (6)  \\
\hline
$\chi_{\mathrm{MonHaf}_6}(g)$ & 15 & 3 & 0 &  3 & 1 & 0 & 0 & 7 & 1 & 3 & 1 \\
\hline

\end{tabular}
\end{center}
\end{table}

Using Table \ref{table:s6hafchar} we have

\begin{equation}\label{eq:s6-3}
|\langle \chi_{\mathrm{MonHaf}_6},\chi_{\mathrm{MonHaf}_6}\rangle|/720=3
\end{equation}
 \end{ex}
Hence by using Equation \ref{eq:s6-3}, Table \ref{table:s6hafchar}, and the character table of  $\mathfrak{S}_6$ we have the following lemma.

\vspace{0.2in}

\begin{lem}\label{lem:rep explanation 6}
 The $\mathfrak{S}_6$ representation of  $\mathrm{MonHaf}_6$ is the sum of three irreducible representations of $\mathfrak{S}_6$, namely to those corresponding to the partitions $[6]$, $[4,2]$ and $[2,2,2]$ of dimensions 1, 9 and 5 respectively. The 10-dimensional image of the map $\Phi$ is the sum of the irreducible representations corresponding to the partitions [6] and [4,2], and the 5-dimensional kernel corresponds to the irreducible representation of partition $[2,2,2]$. 

There is a duality here between $\Phi$ and $\Psi$, since the kernel of $\Phi$ is the image of $\Psi$ after we replace the variables of $R^s$ with the variables of $S^s$, and vice versa.
\end{lem}

 \vspace{.2in}
 In a conversation with Claudiu Raicu, he explained that the representation of $\mathrm{MonHaf}_{2k}$ is the sum of the irreducible representations of $\mathfrak{S}_{2k}$ corresponding to the partitions with all even parts. We checked using a program we wrote  in JAVA the analogous result for $n=8$, namely that $\mathrm{MonHaf}_{8}$ is a direct sum of irreducible representations of $\mathfrak{S}_{8}$ corresponding to the partitions with all even parts. 
 \vspace{.2in}

\begin{thm}\label{thm:Kerber}(\cite{KE}, page 272, Theorem 5.8.3)
For each $n\in \mathbb{N}$ we have the following decomposition of plethysm of identity representations of symmetric group

$$[2]\odot [n]=\sum_{\alpha\vdash n}[2\alpha],$$

if $2\alpha=(2\alpha_1,2\alpha_2,\ldots)$.

\end{thm}

 $\mathrm{MonHaf}_{2m}$ is a subspace of $S_m(S_2V)$ and since it does not involve any diagonal element it can be considered as a subspace of $S_m(\wedge^2V)$. The following Theorem about the $\mathrm{Gl}(V)$ representations of the space $S_m(S_2V)$ seems relevant.

\vspace{0.2in}

\begin{thm}(\cite{We},  page 65)
Let $\sf k$ be a commutative ring of characteristic $0$.

$$S_m(S_2E)=\oplus_{|\lambda |=2m, \lambda'_i \text{ even for all i }} L_\lambda E.$$
$$S_m(\wedge^2E)=\oplus_{|\lambda |=2m,\lambda_i \text{ even for all i }}L_\lambda E.$$
\end{thm}


\end{document}